\documentclass[12pt]{article}
\usepackage{amssymb}
\usepackage{amsmath}
\usepackage{amsthm}
\usepackage{enumitem}
\usepackage{graphicx}
\usepackage[total={6.5in,8.75in}, top=1.2in, left=0.9in, includefoot]{geometry}
\def\bcb{border-collision bifurcation}
\def\cS{\mathcal{S}}
\def\cT{{J}}
\def\cW{\mathcal{W}}

\def\nsf{border-collision fold}
\def\pwl{piecewise-linear}
\def\pws{piecewise-smooth}
\def\rot{rotation number}
\def\sew{symbol sequence}
\def\shr{shrinking point}
\def\sL{{\sf L}}
\def\sR{{\sf R}}
\def\sw{switching manifold}
\def\tong{resonance tongue}

\def\Ther{Thm.}
\def\Lemm{Lemma}
\def\Defi{Def.}

\begin{document}

\newtheorem{theorem}{Theorem}
\newtheorem{corollary}[theorem]{Corollary}
\newtheorem{lemma}[theorem]{Lemma}
\theoremstyle{definition}
\newtheorem{definition}{Definition}

\title{Resonance near Border-Collision Bifurcations in Piecewise-Smooth, Continuous Maps.}
\author{
D.J.W.~Simpson$^{\dagger}$ and J.D.~Meiss$^{\ddagger}$\thanks{
DJWS acknowledges support from an NSERC Discovery Grant.
JDM acknowledges support from NSF grant DMS-0707659.}
\\\\
$^{\dagger}$Department of Mathematics\\
University of British Columbia\\
Vancouver, BC, V6T1Z2\\
Canada\\\\
$^{\ddagger}$Department of Applied Mathematics\\
University of Colorado\\
Boulder, CO, 80309-0526\\
USA}
\maketitle

\begin{abstract}
Mode-locking regions (resonance tongues) formed by
border-collision bifurcations of piecewise-smooth,
continuous maps commonly exhibit a
distinctive sausage-like geometry with pinch points called ``shrinking points".
In this paper we extend our unfolding of the piecewise-linear case
[{\em Nonlinearity}, 22(5):1123-1144, 2009] to show how shrinking points are
destroyed by nonlinearity. We obtain a codimension-three unfolding of this
shrinking point bifurcation for $N$-dimensional maps. We show that the destruction
of the shrinking points generically occurs by the creation of a curve
of saddle-node bifurcations that smooth one boundary of the sausage,
leaving a kink in the other boundary.
\end{abstract}

\section{Introduction}
\label{sec:INTRO}

Piecewise-smooth systems are used to model
a vast range of physical systems involving nonsmooth behavior
\cite{DiBu08,ZhMo03,BaVe01,LeNi04}.
In this paper we study \pws, continuous maps, i.e.
\begin{equation}
x_{i+1} = F(x_i) \;,
\label{eq:genMap}
\end{equation}
where $x_i \in \mathbb{R}^N$ and
$F$ is everywhere continuous but nondifferentiable on
codimension-one surfaces in $\mathbb{R}^N$ called {\em \sw s}.
Such maps arise as Poincar\'{e} maps of Filippov systems
near grazing-sliding and corner collisions \cite{DiKo02,DiBu01c}
or of some time-dependent \pws~flows \cite{ZhMo08b}. They
are often used as mathematical models of various discrete, nonsmooth systems,
see e.g. \cite{PuSu06}.

A fundamental and unique bifurcation of \pws, continuous maps results
from the collision of a fixed point with a \sw; it is known as a {\em \bcb}.
Except in degenerate cases,
a~\bcb~may be classified as either a
{\em \nsf}~at which two fixed points collide and annihilate,
or a {\em border-collision persistence} at which
a single fixed point ``crosses" the \sw~\cite{DiBu08,DiFe99}.
Note that the collision of one point of a periodic solution of (\ref{eq:genMap})
with a \sw~is also a \bcb~for the $n^{\rm th}$ iterate of 
(\ref{eq:genMap}) \cite{DiFe99,SiMe09}.
Though complicated dynamics may be born in \bcb s, in this paper 
we study only the creation of periodic solutions.

We assume that the derivatives of the smooth components of (\ref{eq:genMap}), $D_x F$, are locally bounded (we exclude from consideration, for instance,
square-root type maps \cite{DiBu08,No01}).
Then, generic \bcb s of (\ref{eq:genMap}) may be described by \pwl~maps.
More precisely, structurally-stable dynamics of a \pws, continuous map
near a \bcb~are described by the \pwl, series expansion about the bifurcation.
A consequence is that, to lowest order, the structurally-stable invariant sets 
created at \bcb s grow linearly as the bifurcation parameter varies.

One example is
the two-dimensional, \pws, continuous map
\begin{equation}
f_\mu (x) =
\left\{ \begin{array}{lc}
\mu \left[ \begin{array}{c} 1 \\ 0 \end{array} \right] +
\left[ \begin{array}{cc}
2 r_\sL \cos({2 \pi \omega_\sL}) & 1 \\ -r_\sL^2 & 0 \end{array} \right] x +
g^\sL(x), & s \le 0 \\
\mu \left[ \begin{array}{c} 1 \\ 0 \end{array} \right] +
\left[ \begin{array}{cc}
\frac{2}{s_\sR} \cos({2 \pi \omega_\sR}) & 1 \\ -\frac{1}{s_\sR^2} & 0 \end{array} \right] x +
g^\sR(x), & s \ge 0
\end{array} \right. \;,
\label{eq:DNSex}
\end{equation}
where 
\begin{equation}
s \equiv e_1^{\sf T} x~~({\rm the~first~component~of~the~vector~}
x \in \mathbb{R}^2) \;, \nonumber
\end{equation}
$0 < r_\sL, s_\sR < 1$, $0 < \omega_\sL, \omega_\sR < \frac{1}{2}$,
$\mu \in \mathbb{R}$ is assumed to be small,
and the functions $g^\sL(x)$ and $g^\sR(x)$
contain the terms of $f_\mu$ that are nonlinear in $x$.
%, that is they are $O(|x|^2)$.
The map (\ref{eq:DNSex}) is continuous only if
$g^\sL(x) = g^\sR(x)$ whenever $s = 0$, i.e., on the \sw.

A \bcb~for (\ref{eq:DNSex}) occurs at the origin when $\mu = 0$.
If this bifurcation is nondegenerate, the local dynamics are independent of the
nonlinear components, $g^\sL$ and $g^\sR$.

The piecewise-linear version of (\ref{eq:DNSex}), i.e., with
\begin{equation}
g^\sL(x) = g^\sR(x) =
\left[ \begin{array}{c} 0 \\ 0 \end{array} \right] \;,
\label{eq:nonlinearity0}
\end{equation}
is the canonical, \pwl~form for a border-collision analogue
to a smooth Neimark-Sacker bifurcation.
Indeed the fixed point for $\mu < 0$
has a pair of complex multipliers
($\lambda_{\pm} = r_\sL {\rm e}^{\pm 2 \pi {\rm i} \omega_\sL}$)
inside the unit circle that ``jump'' at $\mu = 0$
outside the unit circle
($\lambda_{\pm} = \frac{1}{s_\sR} {\rm e}^{\pm 2 \pi {\rm i} \omega_\sR}$)
for $\mu > 0$.
Depending upon the precise choice of parameter values,
this \bcb~may generate periodic, quasiperiodic,
or chaotic solutions as well as combinations of these
\cite{SiMe08b,NuYo92,ZhMo06b,SuGa08}.
This class of \bcb s has been seen in
models of DC/DC power converters \cite{ZhMo08b}
and optimization in economics \cite{PuSu06,SuGa04}.

%%%%%%%%%%%%%%%%%%%%%%%%%%%%%%%%%%%%%%%%%%%%%%%%%%%%%%%%%%%%%
\begin{figure}[ht]
\begin{center}
\includegraphics[width=14.7cm,height=6cm]{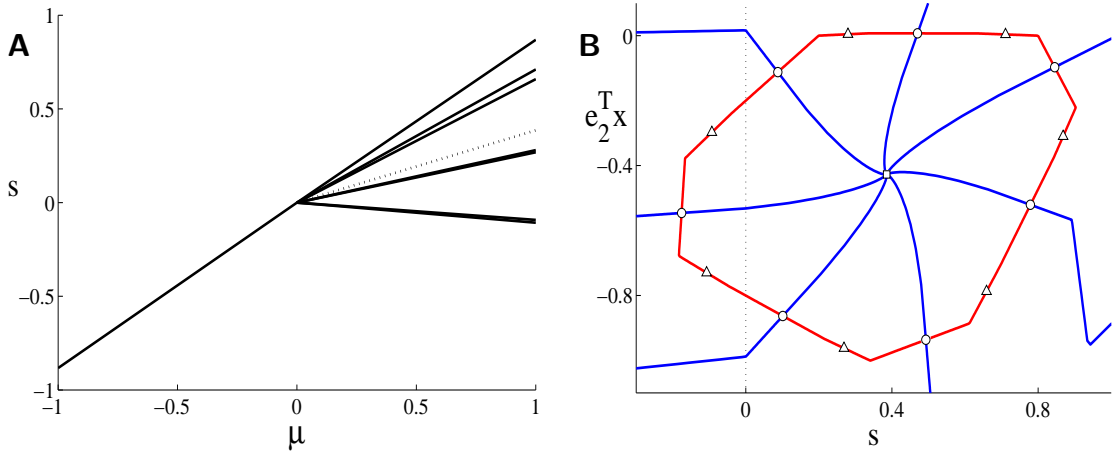}
%\setlength{\unitlength}{1cm}
%\begin{picture}(14.7,6)
%\put(0,0){\includegraphics[width=7.2cm,height=6cm]{bifDiag_2_7b}}
%\put(7.7,0){\includegraphics[width=7.2cm,height=6cm]{manifolds_2_7b}}
%\put(0,5.3){\large \sf \bfseries A}
%\put(7.7,5.3){\large \sf \bfseries B}
%\end{picture}
\caption{
Panel A shows a bifurcation diagram of (\ref{eq:DNSex}) with (\ref{eq:nonlinearity0})
when $r_\sL = 0.2$, $s_\sR = 0.95$ and $\omega_\sL = \omega_\sR = 0.287$.
When $\mu < 0$ the unique fixed point is attracting (the solid line) and
when $\mu > 0$ it is repelling (the dotted line) and
a pair of period-seven orbits exist.
The solid lines for $\mu >0$ in panel A show the attracting 7-cycle
(two pairs of points have similar $s$-values).
Panel B shows a phase portrait when $\mu = 1$. The dotted vertical line is the \sw.
The attracting [saddle] 7-cycle is indicated by triangles [circles].
The unstable manifold of the saddle forms an invariant circle (red);
the curves forming its stable manifold (blue) intersect
at the repelling fixed point (square).
\label{fig:ex_2_7}
}
\end{center}
\end{figure}
%%%%%%%%%%%%%%%%%%%%%%%%%%%%%%%%%%%%%%%%%%%%%%%%%%%%%%%%%%%%%

Fig.~\ref{fig:ex_2_7}-A shows an example of a bifurcation diagram for (\ref{eq:DNSex})
with (\ref{eq:nonlinearity0})
when a pair of period-seven orbits
(7-cycles) is created at $\mu = 0$.
These two orbits are shown in panel B; one is attracting and the other is a saddle.
The unstable manifold of the saddle
forms an invariant circle that contains the attracting orbit.
This \bcb~is nondegenerate; indeed, 
if we were to relax (\ref{eq:nonlinearity0}) then there is 
an $\varepsilon > 0$ such that whenever $0 < \mu < \varepsilon$
the dynamics are topologically equivalent to
Fig.~\ref{fig:ex_2_7}-B near the origin.
An example is shown in Fig.~\ref{fig:bifDiag_2_7c} where
\begin{equation}
g^\sL(x) = \left[ \begin{array}{c} s^2 \\ 0 \end{array} \right] \hspace{20mm}
g^\sR(x) = \left[ \begin{array}{c} 0 \\ 0 \end{array} \right] \;.
\label{eq:nonlinearity1}
\end{equation}
For this nonlinear system, there is also a saddle-node pair of period-seven orbits created
at the \bcb. The stable 7-cycle is attracting up to $\mu \approx 1.202$.

%%%%%%%%%%%%%%%%%%%%%%%%%%%%%%%%%%%%%%%%%%%%%%%%%%%%%%%%%%%%%
\begin{figure}[ht]
\begin{center}
\includegraphics[width=7.2cm,height=6cm]{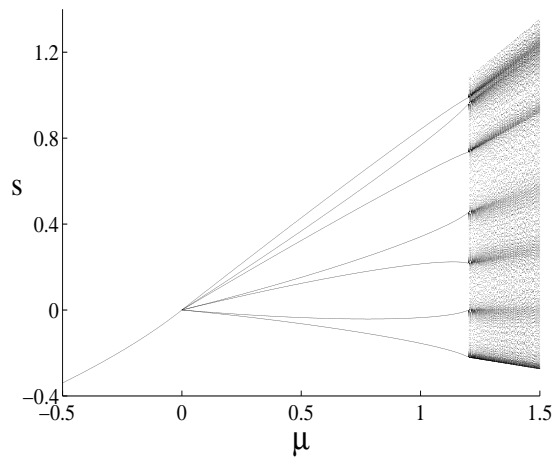}
%\setlength{\unitlength}{1cm}
%\begin{picture}(7.2,6)
%\put(0,0){\includegraphics[width=7.2cm,height=6cm]{bifDiag_2_7c}}
%\end{picture}
\caption{
A bifurcation diagram showing attracting orbits of (\ref{eq:DNSex})
with nonlinearity (\ref{eq:nonlinearity1})
at the same parameter values in Fig.~\ref{fig:ex_2_7}.
When $\mu$ is small the dynamical behavior near the origin
is topologically equivalent to the linear case.
However, at $\mu \approx 1.202$, the attracting 7-cycle undergoes border-collision
beyond which there exists a complicated attracting set.
\label{fig:bifDiag_2_7c}
}
\end{center}
\end{figure}
%%%%%%%%%%%%%%%%%%%%%%%%%%%%%%%%%%%%%%%%%%%%%%%%%%%%%%%%%%%%%

Note that piecewise-linear maps are particularly straightforward to analyze
because any
periodic orbit is the solution to a linear system.
%and the inverse map (if it exists) is easily computable.
Furthermore, linearity implies that if $\mathcal{I}$ is an invariant set of $f_\mu$
then $\lambda \mathcal{I}$ is an invariant set of $f_{\lambda \mu}$
for any $\lambda > 0$.
Consequently it suffices to consider $\mu = -1,0,1$.

A two-parameter bifurcation diagram
of the piecewise-linear case of (\ref{eq:DNSex})
is shown in Fig.~\ref{fig:rrEqrL_2} for $\mu = 1$.
The colored regions are resonance (or Arnold) tongues within
which there is an attracting periodic solution.
Since the piecewise-linear case of the two-dimensional map (\ref{eq:DNSex})
has a unique fixed point, we may define a \rot~for orbits as the average
change in angle per iteration about the
fixed point \cite{SiMe08b, Le01,Fr90,Sc57}.
Consequently the \tong s can be labeled by the \rot, $m/n$,
of the corresponding periodic solution.
The orbits shown in Fig.~\ref{fig:ex_2_7}
lie in the $2/7$-tongue; this tongue intersects the $s_R=1$ line
at $\omega_R = 2/7 \approx 0.2857$.

%%%%%%%%%%%%%%%%%%%%%%%%%%%%%%%%%%%%%%%%%%%%%%%%%%%%%%%%%%%%%
\begin{figure}[ht]
\begin{center}
\includegraphics[width=15cm,height=6cm]{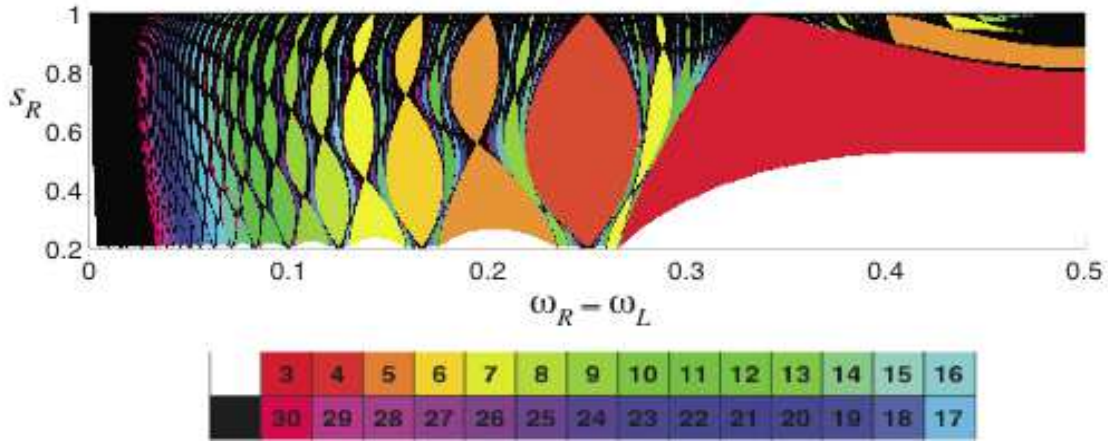}
%\setlength{\unitlength}{1cm}
%\begin{picture}(15,6)
%\put(0,0){\includegraphics[width=15cm,height=6cm]{rrEqrL_2}}
%\end{picture}
\caption{
Resonance tongues of (\ref{eq:DNSex}) with (\ref{eq:nonlinearity0})
for $r_\sL = 0.2$ and $\mu = 1$.
Each colored region corresponds to the existence of a periodic solution
with period indicated in the color bar.
The \rot~in each \tong~is equal to the value of
$\omega_\sR$ at which it emanates from the line, $s_\sR = 1$.
White regions correspond to no observed attractor;
black regions correspond to the existence of a bounded forward orbit
that is either aperiodic or has period larger than 30.
Similar figures at different
parameter values are given in \cite{SiMe08b,Si10}.
\label{fig:rrEqrL_2}
}
\end{center}
\end{figure}
%%%%%%%%%%%%%%%%%%%%%%%%%%%%%%%%%%%%%%%%%%%%%%%%%%%%%%%%%%%%%

The majority of the \tong s in Fig.~\ref{fig:rrEqrL_2}
exhibit a structure that is
often likened to a string of sausages.
This structure
was first observed in a one-dimensional sawtooth map \cite{YaHa87},
and has since been described in higher dimensional maps such as (\ref{eq:DNSex}),
see for example \cite{ZhMo06b,ZhMo08b,ZhSo03,PuSu06,SuGa04}.
As in \cite{YaHa87}, we refer to points where \tong s have zero width as
{\em \shr s}\footnote
{In this paper we will only consider the case of ``non-terminating" 
shrinking points defined in \cite{SiMe09}.
Thus we do not consider the ends of \tong s
such as those at $s_R = 1$ in Fig.~\ref{fig:rrEqrL_2}.}.

An unfolding of \shr s in \pwl, continuous maps of arbitrary dimension
was performed in \cite{SiMe09}.
There it was shown, upon imposing reasonable nondegeneracy assumptions,
that any two-dimensional slice of parameter space in the neighborhood of a
\shr~will resemble Fig.~\ref{fig:shrPoint_2_7}:
\shr s are codimension-two phenomena of \pwl, continuous maps.
In particular, near a \shr, the resonance 
tongue is locally a two-dimensional cone
bounded by four curves. These boundary curves are pairwise tangent
at the \shr~so the cone boundaries are locally $C^1$. 
In the interior of the cone  a {\em primary $n$-cycle} exists; it has some number of
points, say $l$, located, ``left'', of the \sw~($l=2$ for Fig.~\ref{fig:shrPoint_2_7}).
This orbit collides and annihilates with another $n$-cycle
in a \nsf~bifurcation on the four boundary curves of the cone.
This secondary periodic solution has $l-1$ points to the left of the \sw~within one
half of the cone and $l+1$ points to the left of the \sw~within the other half of the cone.
On each of the four boundary curves a different point of the primary orbit
lies on the \sw.

%%%%%%%%%%%%%%%%%%%%%%%%%%%%%%%%%%%%%%%%%%%%%%%%%%%%%%%%%%%%%
\begin{figure}[ht]
\begin{center}
\includegraphics[width=8.4cm,height=7cm]{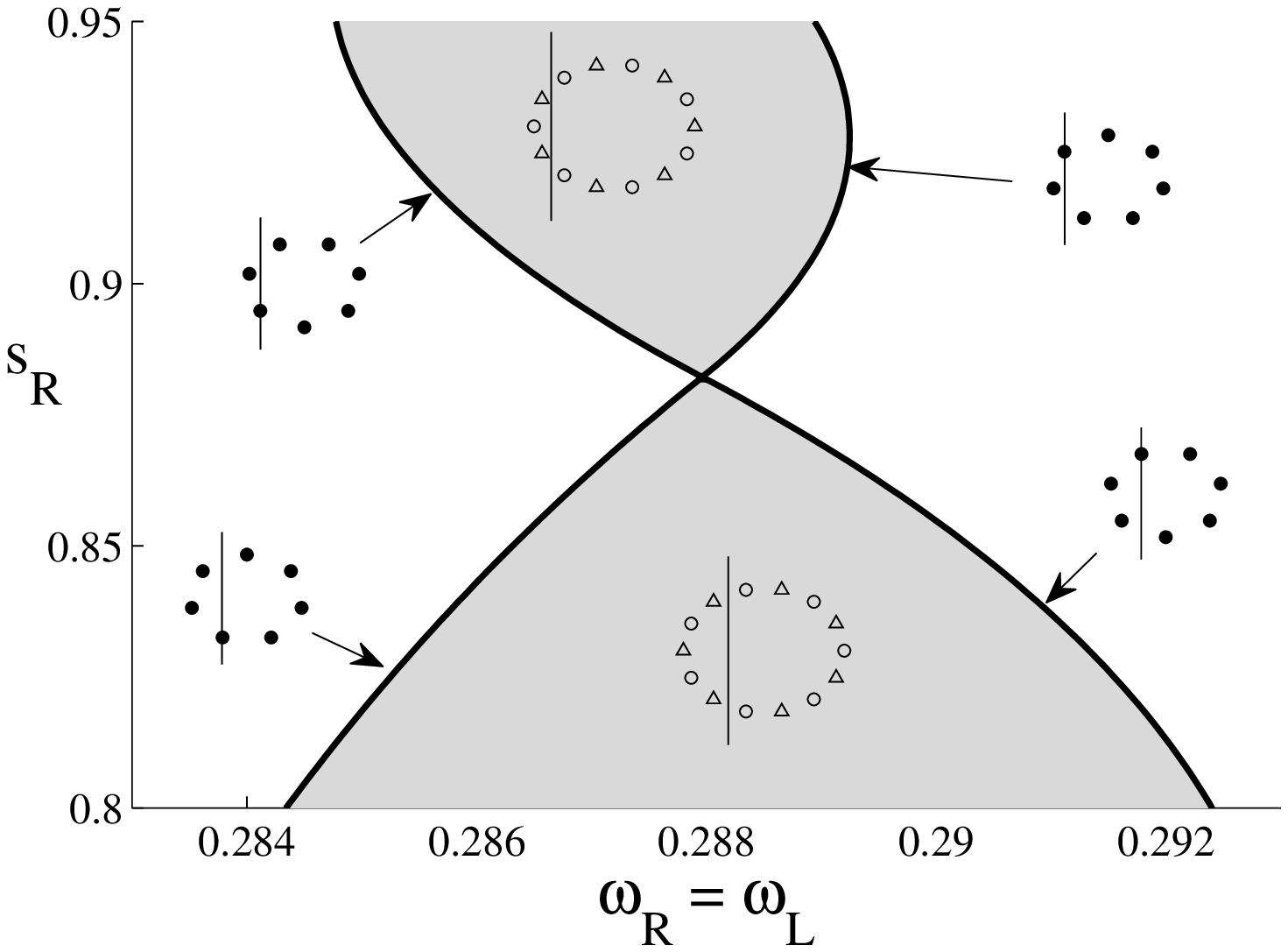}
%\setlength{\unitlength}{1cm}
%\begin{picture}(8.4,7)
%\put(0,0){\includegraphics[width=8.4cm,height=7cm]{shrPoint_2_7}}
%\end{picture}
\caption{
A magnification of the $2/7$ \tong~in Fig.~\ref{fig:rrEqrL_2} near a \shr.
Orbits that exist in the \tong~are shown in inserted phase portraits
(not to scale) as the open circles and triangles.
Above the \shr, these have $l=1$ and $l=2$
($l$ represents the number of points that lie left of the \sw)
and below the \shr, $l=2$ and $l=3$.
On the four boundary curves the two coexisting periodic solutions
collide and annihilate in \nsf~bifurcations.
On each of the four boundary curves a different point of the primary orbit
(with $l=2$), lies on the \sw.
These orbits are sketched in phase portraits with filled circles.
%Triangles [circles] points correspond to a stable [unstable] solution.
Compare this to panel B of Fig.~\ref{fig:ex_2_7}, an actual phase portrait at
$(\omega_R,s_R) = (0.287,0.95)$.
\label{fig:shrPoint_2_7}
}
\end{center}
\end{figure}
%%%%%%%%%%%%%%%%%%%%%%%%%%%%%%%%%%%%%%%%%%%%%%%%%%%%%%%%%%%%%

These results generalize to the nonlinear case in the follwing sense: 
regardless of the nonlinearites $g^\sL$ and $g^\sR$, the
two-dimensional bifurcation structure shown in Fig.~\ref{fig:rrEqrL_2}
will be essentially unchanged when $\mu$ is small enough and it will limit to the 
linear picture as $\mu \to 0^+$.
However, when nonlinear terms are present
the sausage structure is not preserved as $\mu$ increases; indeed,
the shrinking points break apart as shown in Fig.~\ref{fig:manyBreak}.
Consequently when nonlinear terms are present the phenomenon is codimension-three;
we refer this scenario as a {\em generalized \shr}.

%%%%%%%%%%%%%%%%%%%%%%%%%%%%%%%%%%%%%%%%%%%%%%%%%%%%%%%%%%%%%
\begin{figure}[ht]
\begin{center}
\includegraphics[width=13cm,height=10.6cm]{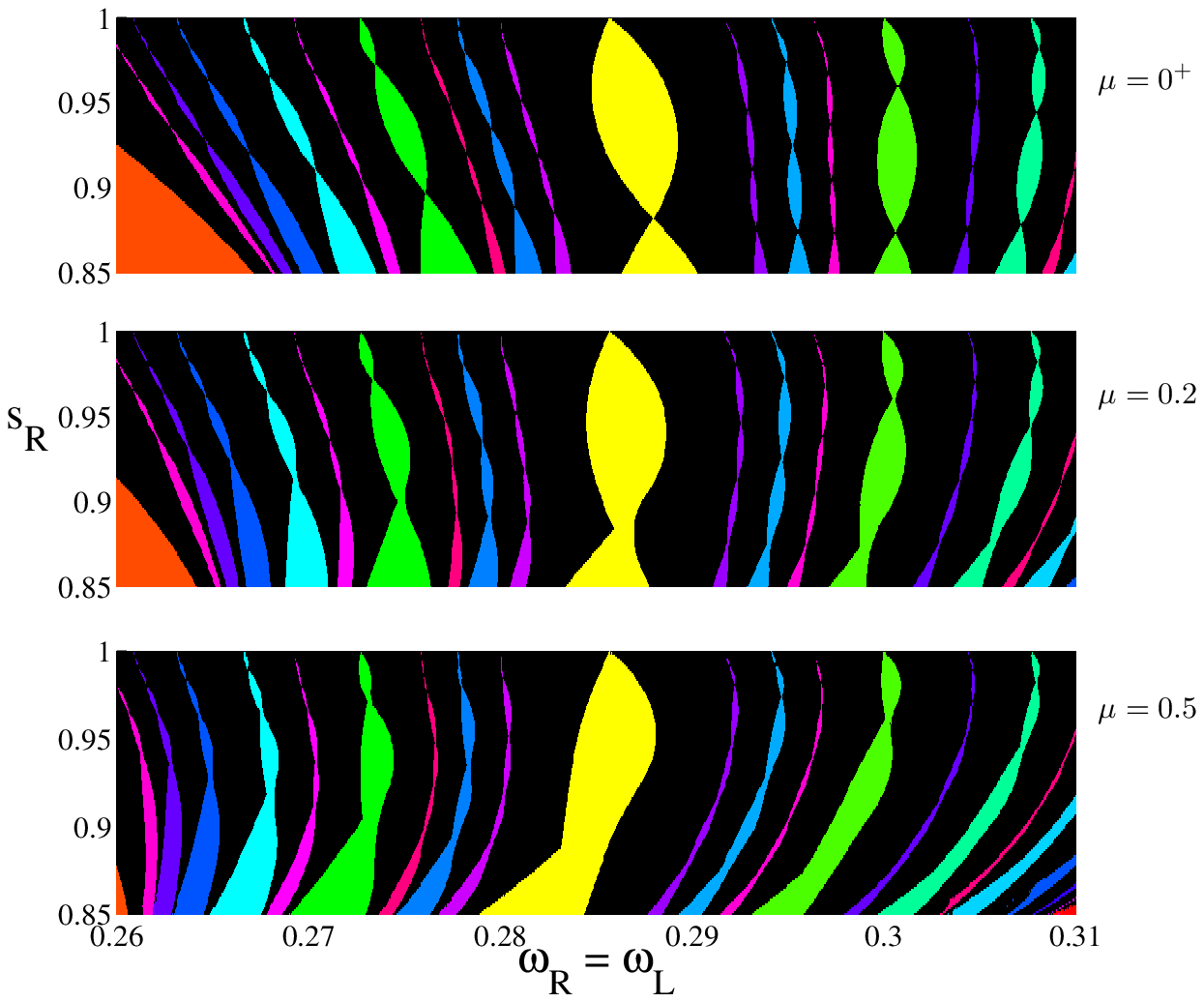}
%\setlength{\unitlength}{1cm}
%\begin{picture}(13,10.6)
%\put(0,7.3){\includegraphics[width=13cm,height=3.3cm]{manyBreak0}}
%\put(0,4){\includegraphics[width=13cm,height=3.3cm]{manyBreak2}}
%\put(0,0){\includegraphics[width=13cm,height=4cm]{manyBreak5}}
%\put(12,9.6){\small $\mu = 0^+$}
%\put(12,6.3){\small $\mu = 0.2$}
%\put(12,3){\small $\mu = 0.5$}
%\end{picture}
\caption{
Resonance tongues of (\ref{eq:DNSex}) with (\ref{eq:nonlinearity1})
when $r_\sL = 0.2$ for three different values of $\mu$.
The uppermost plot is a magnification of Fig.~\ref{fig:rrEqrL_2}.
%and includes the resonance tongue of Fig.~\ref{fig:shrPoint_2_7}.
The lower two plots were computed numerically
by estimating the eventual period of the forward orbit of the origin.
\label{fig:manyBreak}
}
\end{center}
\end{figure}
%%%%%%%%%%%%%%%%%%%%%%%%%%%%%%%%%%%%%%%%%%%%%%%%%%%%%%%%%%%%%

Fig.~\ref{fig:manyBreak} suggests that the break-up of
\shr s occurs in a regular fashion.
The \tong s develop a nonzero width that increases with $\mu$.
For $\mu > 0$ the right boundary of each \tong~appears to be smooth
whereas each left boundary retains the kink that appeared at the \shr.
It is interesting that, in contrast to the case for smooth maps,
we observe no ``strong" resonance behavior when $n < 5$.

The purpose of this paper is to determine the generic unfolding
of generalized \shr s for a \pws, continuous map of arbitrary dimension.
Here we summarize our results.
Suppose that a \bcb~occurs when a parameter $\mu$ is zero, that
resonance arises for $\mu > 0$, and that in the limit $\mu \to 0^+$,
a two-parameter bifurcation diagram exhibits a \shr~with its four boundary curves.
%The \shr~lies at the intersection of four curves
%(described in \cite{SiMe09}) that persist for small $\mu > 0$, see Fig.~\ref{fig:breakup}.
%On each curve a different point of a primary $n$-cycle lies
%on the \sw.
For small $\mu > 0$ the four boundary curves maintain
a common intersection point, $O$,
but each curve is no longer necessarily
tangent to the opposite curve at this point.

We will show that, under reasonable nondegeneracy assumptions, 
there exists a new bifurcation locus for each small enough fixed $\mu > 0$.
This locus is a curve of saddle-node bifurcations of the primary $n$-cycle
that is tangent to one boundary curve at a point $A$, and to an
adjacent boundary curve on the other side 
of the \shr~at $B$, see Fig.~\ref{fig:breakup}.
This locus is smooth and collapses to
the \shr~as $\mu \to 0^+$.
Two $n$-cycles that have the same itinerary as the
primary $n$-cycle exist
in a triangular region bounded by $O$, $A$ and $B$.
If we let $\theta_1$ denote the angle made at $O$
between the two border-collision curves across the region
$AOB$, see Fig.~\ref{fig:breakup},
and $\theta_2$ is the opposing angle made at $O$ between the
other two border-collision curves,
then for small $\mu > 0$, $\theta_1 < \theta_2$.

A formal statement that includes these results is given in
\Ther~\ref{th:unfold} in \S\ref{sec:UNFOLD}.

%%%%%%%%%%%%%%%%%%%%%%%%%%%%%%%%%%%%%%%%%%%%%%%%%%%%%%%%%%%%%
\begin{figure}[ht]
\begin{center}
\includegraphics[width=15.4cm,height=4.2cm]{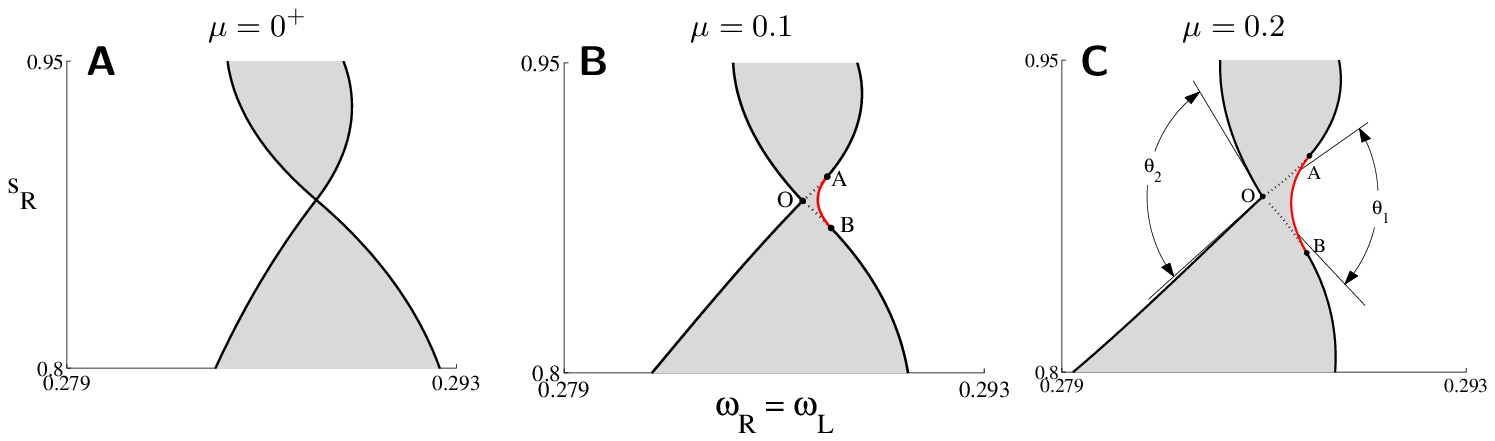}
%\setlength{\unitlength}{1cm}
%\begin{picture}(15.4,4.2)
%\put(0,.3){\includegraphics[width=5.1cm]{breakup1}}
%\put(5,0){\includegraphics[width=5.5cm]{breakup2}}
%\put(10.4,.43){\includegraphics[width=4.7cm]{breakup3}}
%\put(.86,3.7){\large \sf \bfseries A}
%\put(5.86,3.7){\large \sf \bfseries B}
%\put(10.96,3.7){\large \sf \bfseries C}
%\put(2.1,4.1){\small $\mu = 0^+$}
%\put(7,4.1){\small $\mu = 0.1$}
%\put(12,4.1){\small $\mu = 0.2$}
%\end{picture}
\caption{
Bifurcation sets of (\ref{eq:DNSex}) with (\ref{eq:nonlinearity1})
for $r_\sL = 0.2$ and three different $\mu$ values.
Panel A is identical to Fig.~\ref{fig:shrPoint_2_7}.
In panels B and C the locus of classical saddle-node bifurcations
of the primary orbit is indicated by a red curve.
Solid [dotted] curves correspond to \nsf~bifurcations [border-collision persistence].
\label{fig:breakup}
}
\end{center}
\end{figure}
%%%%%%%%%%%%%%%%%%%%%%%%%%%%%%%%%%%%%%%%%%%%%%%%%%%%%%%%%%%%%

Here we summarize the remainder of the paper.
In the following section we present the $N$-dimensional map (\ref{eq:pwsMap})
that describes an arbitrary \bcb.
Concepts from symbolic dynamics that are
invaluable for describing periodic solutions of (\ref{eq:pwsMap}) are introduced
in \S\ref{sec:SYMDYNS}.
Key formulas for periodic solutions are obtained in \S\ref{sec:PERSOLNS}.
Subsequently we impose the assumption
that (\ref{eq:pwsMap}) is piecewise-$C^K$; this 
allows us to derive useful series expansions.
Section \ref{sec:GSP} is devoted to defining
generalized shrinking points as a codimension-three scenario.
Finally in \S\ref{sec:UNFOLD} we unfold these points
leading to \Ther~\ref{th:unfold}.
The proofs of the theorem and \Lemm~\ref{le:expansions} are given in an appendix.

Throughout the paper we use $O(k)$ [$o(k)$] to denote terms that are
order $k$ or larger [larger than order $k$] in all variables
and parameters of a given expression.

%%%%%%%%%%%%%%%%%%
\section{Generic border-collision bifurcations}
\label{sec:BCB}

We restrict our attention to local dynamics of \bcb s, and
thus study a \pws~series expansion of (\ref{eq:genMap})
about an arbitrary \bcb.
Throughout this paper we will assume that  there is a single, smooth \sw~
in a neighborhood of the \bcb.
This is typically the case in models since
\sw s are usually defined by some simple physical constraint.
For simplicity we assume that a coordinate transformation has been made
so that the \sw~corresponds to the vanishing of the first component of $x$ 
%the coordinate plane, $e_1^{\sf T} x = 0$,
%where $x \in \mathbb{R}^N$
(see in particular \cite{DiBu01} for details of such a transformation)
and---to avoid the use of subscripts for components---we introduce the new variable
\begin{equation}
s = e_1^{\sf T} x \;.
\label{eq:sDef}
\end{equation}
A general \pws~map then takes the form
\begin{equation}
x_{i+1} = f(x_i;\xi) = \left\{ \begin{array}{lc}
f^\sL(x_i;\xi) , & s_i \le 0 \\
f^\sR(x_i;\xi) , & s_i \ge 0
\end{array} \right. \;,
\label{eq:pwsMap}
\end{equation}
where $\xi$ is a vector of parameters.
The switching manifold partitions phase space into two regions:
the {\em left half-space} (where $s < 0$)
and the {\em right half-space} (where $s > 0$).

We assume a \bcb~of a fixed point occurs at the origin
when a parameter $\mu$ is zero so that the functions $f^{\sL}$ and $f^{\sR}$
have series expansions
\begin{equation}
f^\cT(x_i;\xi) = \mu b(\xi) + A_\cT(\xi) x_i + g^\cT(x_i;\xi) \;,
\label{eq:fiForm}
\end{equation}
where $\cT = \sL, \sR$,
$\mu$ is the first component the parameters $\xi$,
$A_\cT$ is an $N \times N$ matrix,
$b \in \mathbb{R}^N$. 
The functions $g^\cT$ contain only terms that are nonlinear in $x$;
that is, $g^\cT(x_i;\xi) = o(|x_i|)$ (this, for example, could include terms
of order $|x_i|^{\frac{3}{2}}$ that arise naturally 
in Poincar\'{e} maps relating to sliding bifurcations).
By continuity of (\ref{eq:pwsMap}),
$A_\sL$ and $A_\sR$ are identical in their last $N-1$ columns
and $g^\sL = g^\sR$ whenever $s = 0$.

If $A_\cT(\xi)$ does not have an eigenvalue 1,
the half-map $f^\cT$ has a unique fixed point near the origin
given explicitly by
\begin{equation}
x^{*\cT}(\xi) = (I - A_\cT(\xi))^{-1} b(\xi) \Big|_{\mu = 0} \mu + o(\mu) \;.
\label{eq:xStarT}
\end{equation}
As seems to have been first noted by Feigin, see \cite{DiFe99},
a convenient expression for the first component of the vector, $x^{*\cT}$,
is obtained with adjugate matrices \cite{Be92,Ko96}:
\begin{equation}
{\rm adj}(X) X = \det(X) I \;,~~{\rm for~any~} N \times N {\rm ~matrix~} X.
\label{eq:adjRelation}
\end{equation}
The point is that since $A_\sL$ and $A_\sR$ are identical in their last $N-1$ columns,
${\rm adj}(I - A_\sL)$ and ${\rm adj}(I - A_\sR)$ share the same first row:
\begin{equation}
\varrho^{\sf T}(\xi) = e_1^{\sf T}{\rm adj}(I - A_\sL(\xi))
= e_1^{\sf T}{\rm adj}(I - A_\sR(\xi)) \;.
\label{eq:varrho}
\end{equation}
Consequently, multiplication of (\ref{eq:xStarT}) by $e_1^{\sf T}$ on the left implies
that the first component of the fixed point $x^{*\cT}$ satisfies the useful formula
\begin{equation}
s^{*\cT}(\xi) = \frac{\varrho^{\sf T}(\xi) b(\xi)}{\det(I - A_\cT(\xi))}
\Bigg|_{\mu = 0} \mu + o(\mu) \;.
\label{eq:sStarj}
\end{equation}
In particular we learn from (\ref{eq:sStarj})
that both fixed points, if they exist,
move away from the \sw~linearly with respect to $\mu$ if and only if
$\varrho^{\sf T}(\xi) b(\xi) |_{\mu = 0} \ne 0$.
This condition is a nondegeneracy condition for the \bcb.
Under this assumption, the bifurcation is a border-collision fold bifurcation
if $\det(I-A_\sL(0))$ and $\det(I-A_\sR(0))$ have opposite signs,
and border-collision persistence if they have the same sign \cite{DiFe99}.

%%%%%%%%%%%%%%%%%%%%%%%%%%
\section{Symbolic dynamics}
\label{sec:SYMDYNS}

It is common in the study of \pws~systems for symbolic methods to be
highly beneficial.
In this paper we consider bi-infinite
sequences, $\cS$, constructed from the binary alphabet
$\{ \sL, \sR \}$.
In order to unfold \shr s,
we find it necessary to consider only what we have termed
{\em rotational \sew s}.
In \cite{SiMe09} we defined these as particular finite collections.
Instead of repeating this definition we provide
here a definition that is more versatile
in that it extends naturally to nonperiodic sequences
and has been described elsewhere.
Furthermore, to be consistent with combinatorics literature
 sequences are always assumed to
contain infinitely many elements (unlike in \cite{SiMe09}).

\begin{definition}
For $\alpha, \beta \in [0,1)$,
let $\cS[\alpha,\beta]$, be the bi-infinite \sew~with $i^{th}$ element
\begin{equation}
\cS[\alpha,\beta]_i \equiv \left\{ \begin{array}{lc}
\sL, & i \alpha {\rm ~mod~} 1 \in [0,\beta) \\
\sR, & i \alpha {\rm ~mod~} 1 \in [\beta,1)
\end{array} \right. \;, {\rm ~for~all~} i \in \mathbb{Z}.
\nonumber
\end{equation}
\label{def:abseq}
\end{definition}

%%%%%%%%%%%%%%%%%%%%%%%%%%%%%%%%%%%%%%%%%%%%%%%%%%%%%%%%%%%%%
\begin{figure}[ht]
\begin{center}
\includegraphics[width=8cm,height=7cm]{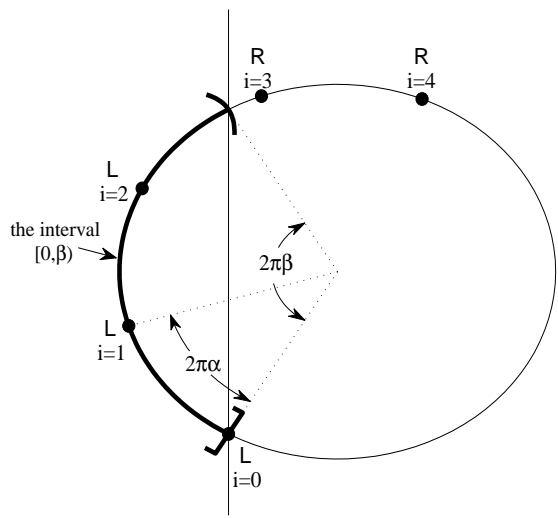}
%\setlength{\unitlength}{1cm}
%\begin{picture}(8.4,7)
%\put(0,0){\includegraphics[width=8cm,height=7cm]{genRSS}}
%\end{picture}
\caption{
A geometric portrayal of \Defi~\ref{def:abseq}.
First cut a circle with a vertical line that subtends
an angle $2\pi\beta$ as shown. 
Each real number, $\phi$, is then represented by a point $2 \pi \phi$ radians clockwise from the
lower intersection of the circle with the vertical line.
Then $\cS_i[\alpha,\beta] = \sL$ whenever
$\phi = {\rm i} \alpha$ is located to the left of the
vertical line, 
%or at the lower intersection point,
and $\cS_i[\alpha,\beta] = \sR$ otherwise.
In addition $\cS_0[\alpha,\beta]$ is always $\sL$.
\label{fig:genRSS}
}
\end{center}
\end{figure}
%%%%%%%%%%%%%%%%%%%%%%%%%%%%%%%%%%%%%%%%%%%%%%%%%%%%%%%%%%%%%

A visual representation of these sequences is provided
by Fig.~\ref{fig:genRSS}.
These sequences seem to have first been studied by Slater \cite{Sl50,Sl67}.
The sequences for which $\alpha = \beta \notin \mathbb{Q}$
have been well studied and
are known as {\em rotation sequences};
they are equivalent to {\em Sturmian sequences} \cite{Fo02},
which are traditionally defined from a combinatorics viewpoint \cite{MoHe40}.
Some discussion of the case $\alpha = \beta \in \mathbb{Q}$
is given in \cite{HaZh98}.
A related approach is to consider arithmetic sequences on $\mathbb{Z}/n\mathbb{Z}$
for some $n \in \mathbb{N}$ \cite{SiTr92}.
Similar sequences for rotational orbits of the H\'enon map are described in \cite{DuMe05}.

\begin{definition}[Rotational Symbol Sequence]
For each $l,m,n \in \mathbb{N}$ with $l,m < n$ and ${\rm gcd}(m,n) = 1$,
$\cS[\frac{m}{n},\frac{l}{n}]$ is a {\em rotational \sew}.
\label{def:rss}
\end{definition}

Any periodic sequence, such as $\cS[\frac{m}{n},\frac{l}{n}]$,
is generated by repeated copies of a finite collection
from $\{ \sL, \sR \}$, termed a {\em word}.
For instance for the word $\cW = \sL \sL \sR$,
the generated sequence is
$\cS = \ldots \sL \sL \sR \sL \sL \sR \sL \sL \sR \dots$.
Here we let $\cW[l,m,n]$ denote the word comprised of
the $i = 0,1,\ldots,n-1$ elements of $\cS[\frac{m}{n},\frac{l}{n}]$.

As an example, let us compute $\cS[\frac{2}{7},\frac{3}{7}]$.
Here $\alpha = \frac{2}{7}$,
thus the numbers $i \alpha {\rm ~mod~} 1$ of \Defi~\ref{def:abseq}
for $i = 0,1,\ldots,n-1$ are
$0, \frac{2}{7}, \frac{4}{7}, \frac{6}{7}, \frac{1}{7},
\frac{3}{7}, \frac{5}{7}$,
and $\beta = \frac{3}{7}$, therefore
$\cW[3,2,7] = \sL \sL \sR \sR \sL \sR \sR$.

Throughout this paper will we use the symbol $d$ to denote the
multiplicative inverse of $m$ modulo $n$,
for example $d=4$ when $m=2$ and $n=7$ as above.
Then
\begin{equation}
\cW[l,m,n]_{id} = \left\{ \begin{array}{lc}
\sL ,& i = 0,\ldots,l-1 \\
\sR ,& i = l,\ldots,n-1
\end{array} \right. \;,
\label{eq:rss}
\end{equation}
where $id$ is taken modulo $n$.
For clarity, throughout this paper we omit ``${\rm mod \,} n$'' where
it is clear modulo arithmetic is being used.

Given a word $\cW$,
we let $\cW^{(i)}$ denote the $i^{\rm th}$ left cyclic permutation of $\cW$
and $\cW^{\overline{i}}$ denote the word that differs from $\cW$
in only the $i^{\rm th}$ element.
For example if
$\cW = \sL \sR \sL \sR \sR$ then
$\cW^{\overline{3}} = \sL \sR \sL \sL \sR$ and
$\cW^{(2)} = \sL \sR \sR \sL \sR$.
If $\cS$ is the sequence generated by $\cW$
we let $\cS^{(i)}$ and $\cS^{\overline{i}}$ denote
the sequences generated by $\cW^{(i)}$ and $\cW^{\overline{i}}$,
respectively.

A key property of \sew s that is verified in \cite{SiMe09} is
\begin{equation}
\cS[{\textstyle \frac{m}{n},\frac{l}{n}}]^{((l-1)d) \overline{0}} =
\cS[{\textstyle \frac{m}{n},\frac{l}{n}}]^{\overline{0} (ld)} \;.
\label{eq:switchS}
\end{equation}

%A word is called {\em primitive} if it cannot be written
%as a power, $\cW^k$, for any $k > 1$.
%$\cW$ is primitive if and only if $\cW \ne \cW^{(i)}$ for all $i \ne 0$.

%%%%%%%%%%%%%%%%%%%%%%%%%%
\section{Periodic solutions}
\label{sec:PERSOLNS}

For any symbol sequence $\cS$, the
iterates of a point $x_0 \in \mathbb{R}^N$
under the two half-maps of (\ref{eq:pwsMap})
in the order determined by $\cS$ are
\begin{equation}
x_{i+1} = f^{(\cS_i)}(x_i;\xi)
\label{eq:altMap}
\end{equation}
In general this may be different from iterating $x_0$ under the map (\ref{eq:pwsMap}).
However, if the sequence $\{ x_i \}$
satisfies the {\em admissibility condition}:
\begin{equation}
\cS_i = \left\{
\begin{array}{lc}
\sL, & {\rm whenever~} s_i < 0 \\
\sR, & {\rm whenever~} s_i > 0
\end{array}
\right.
\label{eq:admissCond}
\end{equation}
for every $i$, then $\{ x_i \}$
coincides with the orbit of $x_0$ under (\ref{eq:pwsMap}).
Notice if $s_i = 0$ there is no restriction on $\cS_i$.
When (\ref{eq:admissCond}) holds for every $i$,
$\{ x_i \}$ is admissible,
otherwise it is virtual.

When $\cS$ has period $n$, the periodic orbits with this sequence
are admissible fixed points of the map
\begin{equation}
f^\cS = f^{\cS_{n-1}} \circ \cdots \circ f^{\cS_0} \;.
\label{eq:fSdef}
\end{equation}
Some straightforward algebra leads to
\begin{equation}
f^\cS(x;\xi) = P_\cS(\xi) b(\xi) \Big|_{\mu = 0} \mu + o(\mu)
+ \left( M_\cS(\xi) \Big|_{\mu = 0} + o(\mu^0) \right) x + o(x) \;,
\label{eq:fSexpansion}
\end{equation}
where
\begin{eqnarray}
M_\cS & = & A_{\cS_{n-1}} \ldots A_{\cS_0} \;, \label{eq:stabMatrix} \\
P_\cS & = &
I + A_{\cS_{n-1}} + A_{\cS_{n-1}} A_{\cS_{n-2}} + \cdots
+ A_{\cS_{n-1}} \ldots A_{\cS_1}  \;. \label{eq:bcMatrix}
\end{eqnarray}
%As in \cite{SiMe09} we refer to $M_\cS$ as the {\em stability matrix} of $\cS$
%and $P_\cS$ as the {\em border-collision matrix} of $\cS$.
Notice $P_\cS$ is independent of $\cS_0$, thus
\begin{equation}
P_\cS = P_{\cS^{\overline{0}}} \;.
\label{eq:PindepS0}
\end{equation}

We now present six lemmas relating to periodic solutions
that are useful for analyzing \shr s.
A consequence of the following lemma is that
$M_\cS$ and $M_{\cS^{\overline{0}}}$
are identical in their last $N-1$ columns.

\begin{lemma}
For any $N \times N$ matrix, $X$,
$X A_\sR = X A_\sL + \zeta e_1^{\sf T}$, for some $\zeta \in \mathbb{R}^N$.
\label{le:firstColumnEqual}
\end{lemma}

\begin{proof}
Since $A_\sL$ and $A_\sR$ are identical in all but possibly their first columns,
we may write $A_\sR = A_\sL + \hat{\zeta} e_1^{\sf T}$ for some $\hat{\zeta} \in \mathbb{R}^N$.
This proves the result with $\zeta = X \hat{\zeta}$.
\hfill
\end{proof}

If $x_0$ is a fixed point of $f^\cS$ then the $n$ points
$\{ x_0,\ldots,x_{n-1} \}$
(where $x_i$ is defined by (\ref{eq:altMap})) describe a periodic solution
(which may be virtual)
that we will refer to as an {\em $\cS$-cycle}.
The following two lemmas are generalizations of those given in \cite{SiMe09}.

\begin{lemma}
Suppose $\{ x_i \}$ is an $\cS$-cycle and $x_j$ lies on the \sw, for some $j$.
Then $\{ x_i \}$ is also an $\cS^{\overline{j}}$-cycle.
\label{le:solvesAlso}
\end{lemma}

\begin{proof}
By continuity:
$f^\sL(x_j;\xi) = f^\sR(x_j;\xi)$, hence there is no restriction on
the $j^{\rm th}$ element of $\cS$.
\hfill
\end{proof}

\begin{lemma}
Suppose $\{ x_i \}$ is an $\cS$-cycle.
Then for any $j$,
$\det \big( I - D_x f^{\cS^{(j)}}(x_j;\xi) \big) =
\det \big( I - D_x f^\cS(x_0;\xi) \big)$.
\label{le:cyclicDetDx}
\end{lemma}

\begin{proof}
By the chain rule the Jacobian $D_x f^\cS(x_0;\xi)$
may be written as a product of matrices:
%\begin{equation}
%D_x f^\cS(x_0;\xi) =
%D_x f^{\cS_{n-1}}(x_{n-1};\xi)
%D_x f^{\cS_{n-2}}(x_{n-2};\xi) \cdots
%D_x f^{\cS_0}(x_0;\xi) \;, \nonumber
%\end{equation}
\begin{equation}
D_x f^\cS(x_0;\xi) =
\prod_{i=n-1}^0 D_x f^{\cS_i}(x_i;\xi) \;. \nonumber
\end{equation}
The spectrum of this product is unchanged if the $N$ matrices
are multiplied in an order that differs only cyclically from this one
(refer to the proof of \Lemm~4 of \cite{SiMe09}),
%e.g.
%\begin{equation}
%D_x f^{\cS_{j-1}}(x_{j-1};\xi)
%D_x f^{\cS_{j-2}}(x_{j-2};\xi) \cdots
%D_x f^{\cS_j}(x_j;\xi) =
%D_x f^{\cS^{(j)}}(x_j;\xi) \;. \nonumber
%\end{equation}
which proves the result.
\hfill
\end{proof}

Whenever $M_\cS(\xi)$ does not have an eigenvalue 1,
the implicit function theorem implies that $f^\cS$ has a unique
fixed point near the origin, call it $x^{*\cS}$.
Using (\ref{eq:fSexpansion}) we obtain
\begin{equation}
x^{*\cS}(\xi) = (I - M_\cS(\xi))^{-1} P_\cS(\xi) b(\xi) \Big|_{\mu = 0} \mu + o(\mu) \;.
\label{eq:xStarS}
\end{equation}
We may derive a formula for the first component of $x^{*\cS}$,
denoted $s^{*\cS}$,
in the same spirit as (\ref{eq:sStarj}) for fixed points of (\ref{eq:pwsMap}).
Let
\begin{equation}
\varrho_\cS^{\sf T}(\xi) = e_1^{\sf T} {\rm adj}(I-M_\cS(\xi)) \;,
\label{eq:varrhoS}
\end{equation}
then
\begin{equation}
s^{*\cS}(\xi) = \frac{\varrho_\cS^{\sf T}(\xi) P_\cS(\xi) b(\xi)}{\det(I-M_\cS(\xi))}
\Bigg|_{\mu = 0} \mu + o(\mu) \;.
\label{eq:sStarSalt}
\end{equation}
As in \cite{Si10} we use two lemmas to derive a convenient formula for $s^{*\cS}$.

\begin{lemma}
The matrices $P_\cS (I-A_{\cS_0})$
and $I-M_\cS$ can differ in only their first columns.
\label{le:ncssMatDiff}
\end{lemma}

\begin{proof}
Using (\ref{eq:stabMatrix}) and (\ref{eq:bcMatrix}),
\begin{eqnarray}
P_\cS (I-A_{\cS_0}) & = & (I + A_{\cS_{n-1}} + A_{\cS_{n-1}} A_{\cS_{n-2}} + \cdots
+ A_{\cS_{n-1}} \ldots A_{\cS_1})(I-A_{\cS_0}) \nonumber \\
& & \hspace{10mm} {\rm expand~and~group~terms~differently \hspace{-1mm}:} \nonumber \\
& = & I-M_\cS + (A_{\cS_{n-1}}-A_{\cS_0}) + A_{\cS_{n-1}}(A_{\cS_{n-2}}-A_{\cS_0}) + \cdots \nonumber \\
& & \hspace{60mm} +~A_{\cS_{n-1}} \ldots A_{\cS_2}(A_{\cS_1}-A_{\cS_0}) \nonumber \\
& & \hspace{10mm} {\rm apply~\Lemm~\ref{le:firstColumnEqual} \hspace{-1mm}:} \nonumber \\
& = & I-M_\cS + \zeta_{n-1} e_1^{\sf T} + \zeta_{n-2} e_1^{\sf T} + \cdots
+ \zeta_1 e_1^{\sf T} \nonumber \\
& & \hspace{10mm} {\rm where~each~\zeta_i \in \mathbb{R}^N,} \nonumber \\
& = & I-M_\cS + \left( \sum_{i=1}^{n-1} \zeta_i \right) e_1^{\sf T} \;. \nonumber
\end{eqnarray}
\end{proof}

\begin{lemma}
$\varrho_\cS^{\sf T} P_\cS = \det(P_\cS) \varrho^{\sf T}$,
where $\varrho_\cS^{\sf T}$ is given by (\ref{eq:varrhoS})
and $\varrho^{\sf T}$ is given by (\ref{eq:varrho}).
\label{le:generalizedsFormula}
\end{lemma}

\begin{proof}
By \Lemm~\ref{le:ncssMatDiff} we have
\begin{eqnarray}
e_1^{\sf T} {\rm adj}(P_\cS (I-A_{\cS_0})) & = & e_1^{\sf T} {\rm adj}(I-M_\cS)
= \varrho_\cS^{\sf T} \nonumber \\
\Rightarrow~~e_1^{\sf T} {\rm adj}(I-A_{\cS_0}) {\rm adj}(P_\cS) & = &
\varrho_\cS^{\sf T}~~({\rm since~}{\rm adj}(XY) = {\rm adj}(Y)
{\rm adj}(X){\rm ~for~any~} X,Y) \nonumber \\
\Rightarrow~~\varrho^{\sf T} {\rm adj}(P_\cS) & = &
\varrho_\cS^{\sf T}~~{\rm by~(\ref{eq:varrho})} \nonumber \\
\Rightarrow~~\det(P_\cS) \varrho^{\sf T} & = &
\varrho_\cS^{\sf T} P_\cS~~{\rm by~(\ref{eq:adjRelation})} \nonumber
\end{eqnarray}
\end{proof}

If $I-M_\cS(\xi)$ is nonsingular,
by \Lemm~\ref{le:generalizedsFormula} and (\ref{eq:sStarSalt}),
\begin{equation}
s^{*\cS}(\xi) = \frac{\det(P_\cS(\xi))}{\det(I-M_\cS(\xi))} \varrho^{\sf T}(\xi) b(\xi)
\Bigg|_{\mu = 0} \mu + o(\mu) \;.
\label{eq:sStarS}
\end{equation}
This expression relates the linear component of
$s^{*\cS}(\xi)$ simply in terms of $\varrho^{\sf T} b$
(which appears in the fixed point equation (\ref{eq:sStarj}))
and the determinants of $P_\cS$ and $I-M_\cS$.
(Feigin's result concerning the creation of 2-cycles at \bcb s
(see \cite{DiFe99}) follows from (\ref{eq:sStarS})
by substituting $\cS = \sL \sR$ and $\cS = \sR \sL$).

We finish this section with an important lemma
that is most simply stated for (\ref{eq:pwsMap}) in the absence of nonlinear terms,
for then all $o(\mu)$ terms given above vanish.
Though this result is proved in \cite{SiMe09},
with the use of \Lemm~\ref{le:generalizedsFormula}
we are now able to provide a pithier proof.

\begin{lemma}
Suppose the map (\ref{eq:pwsMap}) is \pwl, that is $g^\sL = g^\sR = 0$,
and assume $\mu \ne 0$ and $\varrho^{\sf T} b \ne 0$.
\begin{enumerate}[label=\roman{*}),ref=\roman{*}]
\item
\label{it:ScycleBC}
If $I-M_\cS$ is nonsingular, then the unique fixed point of $f^\cS$,
$x^{*\cS}$, given by (\ref{eq:xStarS}),
lies on the \sw~if and only if $P_\cS$ is singular.
\item
\label{it:ScycleSN}
If $P_\cS$ is nonsingular, then $f^\cS$ has a fixed point
if and only if $I-M_\cS$ is nonsingular.
\end{enumerate}
\label{le:BCSN}
\end{lemma}

\begin{proof}
Part (\ref{it:ScycleBC}) follows immediately from (\ref{eq:sStarS}).
If $I-M_\cS$ is nonsingular,
part (\ref{it:ScycleSN}) is trivial by (\ref{eq:xStarS}).
Suppose $I-M_\cS$ is singular and suppose for a contradiction
$f^\cS$ has a fixed point $x^{*\cS}$.
Then by (\ref{eq:fSexpansion}) we have
\begin{equation}
(I-M_\cS) x^{*\cS} = P_\cS b \mu \;.
\nonumber
\end{equation}
Multiplication of this by $\varrho_\cS^{\sf T}$ (\ref{eq:varrhoS}) on the left yields
\begin{equation}
\det(I-M_\cS) s^{*\cS} = \det(P_\cS) \varrho^{\sf T} b \mu \;,
\nonumber
\end{equation}
where we have also used
(\ref{eq:adjRelation}) and \Lemm~\ref{le:generalizedsFormula}.
This provides a contradiction because the left hand side of the previous equation
is zero, whereas by assumption the right hand side is nonzero.
\hfill
\end{proof}

%%%%%%%%%%%%%%%%%%%%%%%%%%
\section{Generalized shrinking points}
\label{sec:GSP}

At this stage we find it useful to impose the extra
assumption that the map (\ref{eq:pwsMap}) under investigation here
is piecewise-$C^K$, for some $K \in \mathbb{N}$.
This allows us to derive series expansions of smooth components
of the map and iterates of the map.
In particular this assumption allows us to apply the center
manifold theorem necessary for proving the
existence of saddle-node bifurcations in \S\ref{sec:UNFOLD}.

In order to unfold a generalized \shr, we must first
give a precise definition of such a point.
We use the results of the previous section to write down
assumptions that guarantee the existence of a periodic solution
with two points on the \sw.
As in \cite{SiMe09}, it is useful to assume that
this periodic solution is admissible.
To state this assumption we need to renormalize the map (\ref{eq:pwsMap}).

Recall that the map (\ref{eq:fiForm}) depends upon 
an arbitrary vector of parameters $\xi$, and that $\mu$ denotes
the first component of this vector. Scaling $x$ by $\mu$ gives
a new map $h$ defined through
\begin{equation}
f^\cT(\mu z;\xi) = \mu h^\cT(z;\xi) \;,
\nonumber
\end{equation}
where $z \in \mathbb{R}^N$. Note that when $f^\cT$ is $C^K$ then
$h^\cT$ is $C^{K-1}$, and using (\ref{eq:fiForm}), it has the expansion
\begin{equation}
h^\cT(z;\xi) = b(\xi) + A_\cT(\xi) z + O(\mu) \;.
\nonumber
\end{equation}
For $\mu \ge 0$, the {\em renormalized} map for $z$ is then
\begin{equation}
z_{n+1} = h(z_n;\xi) = \left\{ \begin{array}{lc}
h^\sL(z_n;\xi) , & u_n \le 0 \\
h^\sR(z_n;\xi) , & u_n \ge 0
\end{array} \right. \;,
\label{eq:renormMap}
\end{equation}
where
\begin{equation}
u = e_1^{\sf T} z \;.
\nonumber
\end{equation}

%We refer to (\ref{eq:renormMap}) as the renormalization of (\ref{eq:pwsMap}).
%for $\mu \ge 0$
Whenever $\mu \ge 0$, if $x = \mu z$ then $f(x;\xi) = \mu h(z;\xi)$.
For any $\cS$-cycle  (\ref{eq:xStarS}), we can also let
$x^{*\cS}(\xi) = \mu z^{*\cS}(\xi)$ so that
$z^{*\cS}(\xi)$ is $C^{K-1}$ and is a fixed point of
$h^\cS = h^{\cS_{n-1}} \circ \cdots \circ h^{\cS_0}$.
Since (\ref{eq:renormMap}) is a ``blow-up'' of phase space,
points $z \in \mathbb{R}^N$ are not necessarily near the origin when $\mu$ is small.
The renormalization effectively transfers the $\mu$-dependence of the
\pws~map from the constant term to the nonlinear terms and this scaling is often
helpful in the analysis.
Note that (\ref{eq:renormMap}) is has nontrivial dynamics for $\mu = 0$; indeed in this 
case it is \pwl~and identical to the
map studied in \cite{SiMe09}.

We now use \Lemm~\ref{le:BCSN} to guarantee the existence of 
a periodic solution with two points on the \sw~in terms
of singularity of matrices, $P_{\cS^{(i)}}$ (\ref{eq:bcMatrix}).

\begin{definition}[Generalized Shrinking Point]
Consider the map (\ref{eq:pwsMap}) with $N \ge 2$
and suppose $\varrho^{\sf T}(0) b(0) \ne 0$.
Let $\cS = \cS[\frac{m}{n},\frac{l}{n}]$
be a rotational \sew~with $2 \le l \le n-2$.
Suppose
\begin{equation}
P_\cS(0) {\rm ~and~} P_{\cS^{((l-1)d)}}(0) {\rm ~are~singular}
\hspace{10mm} {\rm (the~singularity~condition).}
\nonumber
\end{equation}
Let
\begin{eqnarray}
\check{\cS} & = & \cS^{\overline{0}} \;, \label{eq:checkS} \\
\hat{\cS} & = & \cS^{\overline{ld}} \;, \label{eq:hatS}
\end{eqnarray}
and assume $I-M_{\check{\cS}}(0)$ and $I-M_{\hat{\cS}}(0)$ are nonsingular.
Let $\{ \check{x}_i(\xi) \}$ be the unique $\check{\cS}$-cycle near the origin
($\check{x}_0(\xi)$ is given by (\ref{eq:xStarS})).
Let
\begin{equation}
y_i = \check{z}_i(0) \;,
\label{eq:yidef}
\end{equation}
where $\mu \check{z}_i(\xi) = \check{x}_i(\xi)$.
If the orbit $\{ y_i \}$ is admissible
then we say that (\ref{eq:pwsMap}) is at a {\em generalized \shr}~when $\xi = 0$.
\label{def:gsp}
\end{definition}

The sequences $\check{\cS}$ (\ref{eq:checkS}) and $\hat{\cS}$ (\ref{eq:hatS})
are rotational with one less and one more $\sL$ than 
$\cS = \cS[\frac{m}{n},\frac{l}{n}]$, respectively
(specifically, $\check{\cS} = \cS[\frac{m}{n},\frac{l-1}{n}]^{(-d)}$ and
$\hat{\cS} = \cS[\frac{m}{n},\frac{l+1}{n}]$).
The periodic solution $\{ y_i \}$ is fundamental to the shrinking point.
As stated in the following lemma,
it has two points on the \sw.
Moreover if $\{ \hat{x}_i(\xi) \}$ denotes the unique $\hat{\cS}$-cycle near the origin,
then $\hat{z}_i(0) = \check{z}_i(0) = y_i$.
We let
\begin{equation}
t_i = e_1^{\sf T} y_i \;.
\label{eq:tidef}
\end{equation}

\begin{lemma}%[\cite{SiMe09}]
Suppose (\ref{eq:pwsMap}) is at a generalized \shr~when $\xi = 0$.
Then,
\begin{enumerate}[label=\roman{*}),ref=\roman{*}]
\item
\label{it:t0tld}
$t_0 = t_{ld} = 0$;
\item
\label{it:tneighbors}
$t_d, t_{(l-1)d} < 0$, $t_{-d}, t_{(l+1)d} > 0$;
\item
\label{it:yiperiodn}
$\{ y_i \}$ has minimal period $n$;
\end{enumerate}
\label{le:shrPoint}
\end{lemma}

See \cite{SiMe09} for a proof.
Lemma \ref{le:shrPoint} essentially states that the orbit
$\{ y_i \}$ appears as in Fig.~\ref{fig:ySchem}.

A consequence of \Lemm~\ref{le:shrPoint} is that several important
matrices are singular.
To see this, first note that
the point $y_0$ is a fixed point of $h^\cS(y;0)$.
By \Lemm~\ref{le:solvesAlso}
and \Lemm~\ref{le:shrPoint}(\ref{it:t0tld}),
$y_0$ is also a fixed point of $h^{\cS^{\overline{0}\overline{ld}}}(y;0)$.
Using (\ref{eq:rss}), $\cS^{\overline{0}\overline{ld}} = \cS^{(-d)}$,
and so $y_d$ is a fixed point of $h^\cS(y;0)$.
The points $y_0$ and $y_d$ are distinct
(by \Lemm~\ref{le:shrPoint}(\ref{it:yiperiodn})
and since $\{ y_i \}$ is admissible),
in other words there are multiple $\cS$-cycles.
Hence the matrix $I-M_\cS(0)$ must be singular and
consequently each $P_{\cS^{(i)}}(0)$ is singular by
\Lemm~\ref{le:BCSN}(\ref{it:ScycleSN})
producing the following result (given also in \cite{SiMe09}):

\begin{corollary}
Suppose (\ref{eq:pwsMap}) is at a generalized \shr~when $\xi = 0$.
Then,
\begin{enumerate}[label=\roman{*}),ref=\roman{*}]
\item
\label{it:MSsing}
$I-M_\cS(0)$ is singular;
\item
\label{it:PSising}
$P_{\cS^{(i)}}(0)$ is singular, for all $i$.
\end{enumerate}
\label{cor:singMat}
\end{corollary}

For reader convenience let us briefly summarize symbols used:
\begin{eqnarray}
x \in \mathbb{R}^N \;, & \hspace{20mm} &
s = e_1^{\sf T} x \;, \nonumber \\
y_i = \check{z}_i(0) = \hat{z}_i(0) \;, & \hspace{20mm} &
t = e_1^{\sf T} y \;, \nonumber \\
\mu z = x \;, & \hspace{20mm} &
u = e_1^{\sf T} z \;. \nonumber
\end{eqnarray}

%%%%%%%%%%%%%%%%%%%%%%%%%%%%%%%%%%%%%%%%%%%%%%%%%%%%%%%%%%%%%
\begin{figure}[ht]
\begin{center}
\includegraphics[width=8cm,height=6cm]{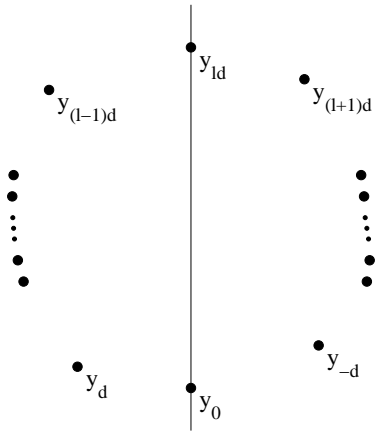}
%\setlength{\unitlength}{1cm}
%\begin{picture}(8,6)
%\put(0,0){\includegraphics[width=8cm,height=6cm]{ySchem}}
%\end{picture}
\caption{
The orbit $\{ y_i \}$ as described by \Lemm~\ref{le:shrPoint}.
The points $y_0$ and $y_{ld}$ lie on the \sw, $s=0$.
\label{fig:ySchem}
}
\end{center}
\end{figure}
%%%%%%%%%%%%%%%%%%%%%%%%%%%%%%%%%%%%%%%%%%%%%%%%%%%%%%%%%%%%%

%%%%%%%%%%%%%%%%%%%%%%%%%%
\section{Unfolding generalized shrinking points}
\label{sec:UNFOLD}

We begin by performing a change of coordinates,
similar to that in \cite{SiMe09},
such that, locally, two boundaries of the associated \tong~lie
on coordinate planes.
We are given that $\check{u}_0(\xi)$ and $\check{u}_{ld}(\xi)$ are $C^{K-1}$
and $\check{u}_0(0) = \check{u}_{ld}(0) = 0$
(\Lemm~\ref{le:shrPoint}(\ref{it:t0tld})).
Since a generalized \shr~is a codimension-three phenomenon,
we assume there are three bifurcation parameters
\begin{equation}
\xi = (\mu,\eta,\nu) \;. \nonumber 
\end{equation}
As long as the matrix
\begin{equation} 
\left. \left[ \begin{array}{cc}
\displaystyle{\frac{\partial \check{u}_0}{\partial \eta}} &
\displaystyle{\frac{\partial \check{u}_0}{\partial \nu}} \\
\displaystyle{\frac{\partial \check{u}_{ld}}{\partial \eta}} &
\displaystyle{\frac{\partial \check{u}_{ld}}{\partial \nu}}
\end{array} \right] \right|_{\xi = 0} \;, \nonumber
\end{equation}
is nonsingular, we may perform a nonlinear coordinate change such that
\begin{eqnarray}
\check{u}_0(\xi) & = & \eta (1 + O(1)) \;, \label{eq:cu_0} \\
\check{u}_{ld}(\xi) & = & \nu (1 + O(1)) \;. \label{eq:cu_ld}
\end{eqnarray}
Consequently, on the coordinate plane $\eta = 0$,
the point $\check{x}_0$ of the $\check{\cS}$-cycle lies on the \sw.
According to \Lemm~\ref{le:solvesAlso}, the $\check{\cS}$-cycle
is also an $\check{\cS}^{\overline{0}} = \cS$-cycle here.
Similarly on $\nu = 0$, $\check{x}_{ld}$ lies on the \sw.
By \Lemm~\ref{le:solvesAlso}, here the $\check{\cS}$-cycle
is also an $\check{\cS}^{\overline{ld}} = \cS^{(-d)}$-cycle
(this equality follows simply from (\ref{eq:rss})).
Along the $\mu$-axis (in three-dimensional parameter space)
both $\check{x}_0$ and $\check{x}_{ld}$ lie on the \sw~so here
the $\check{\cS}$-cycle is also an $\hat{\cS}$-cycle. Hence on the $\mu$-axis
the orbits $\{ \check{x}_i(\xi) \}$ and $\{ \hat{x}_i(\xi) \}$ are identical.

In order for $\xi$ to properly unfold a generalized \shr~we need a
nondegeneracy condition on the nonlinear terms of (\ref{eq:pwsMap})
that guarantees the \shr~breaks apart in the usual manner
as $\mu$ increases from 0.
Along the $\mu$-axis, $\check{x}_0$ is a fixed point of $f^\cS$ and when
$\mu = 0$ it has an associated multiplier of 1.
We will assume that the algebraic multiplicity of this multiplier is one.
An appropriate condition
on the nonlinear terms is that this multiplier varies linearly
(to lowest order) with respect to $\mu$, for small $\mu$.
That is, we require
\begin{equation}
\frac{\partial}{\partial \mu} \det \left( I - D_x f^\cS(\check{x}_0(\xi);\xi) \right)
\Big|_{\xi = 0} \ne 0 \;.
\label{eq:m1}
\end{equation}
We are free to scale the parameter $\mu$ before performing the analysis.
However, we find it is most instructive to merely fix the sign
of $\mu$ in a way that ensures resonance arises for $\mu > 0$.
It turns out that the following assumption achieves this effect:
\begin{equation}
{\rm sgn} \left( \frac{\partial}{\partial \mu}
\det \left(I - D_x f^\cS(\check{x}_0(\xi);\xi) \right)
\Big|_{\xi = 0} \right) = {\rm sgn}(\det(I-M_{\check{\cS}}(0))) \;.
\label{eq:muSign}
\end{equation}

Our unfolding theorem details the existence of admissible periodic solutions
in regions of three-dimensional parameter space that are bounded by
six different surfaces.
As a consequence of the choice (\ref{eq:cu_0}) and (\ref{eq:cu_ld}),
three of these surfaces are simply the coordinate planes.
The following lemma
gives series expansions of the remaining three surfaces.

\begin{lemma}
Suppose the piecewise-$C^K$ map (\ref{eq:pwsMap}) is at a
generalized \shr~when $\xi = 0$ and that $K \ge 4$.
Assume that the only eigenvalue of $M_\cS(0)$ on the unit circle is $1$
and that it has algebraic multiplicity one, and that
(\ref{eq:cu_0}), (\ref{eq:cu_ld}) and (\ref{eq:muSign}) hold.
Let
\begin{eqnarray}
\check{\delta} & = & \det(I-M_{\check{\cS}}(0)) \;, \label{eq:dCheck} \\
\hat{\delta} & = & \det(I-M_{\hat{\cS}}(0)) \;, \label{eq:dHat} \\
k_0 & = & \frac{\partial}{\partial \mu}
\det(I - D_x f^\cS(\check{x}_0(\xi);\xi)) \Big|_{\xi = 0} \;, \label{eq:k0}
\end{eqnarray}
(which are all nonzero due to assumptions
in the definition of a generalized \shr).
Then,
\begin{enumerate}[label=\roman{*}),ref=\roman{*}]
\item
\label{it:ddFormula}
$\displaystyle \frac{\check{\delta}}{\hat{\delta}} =
-\frac{t_d t_{(l-1)d}}{t_{-d} t_{(l+1)d}}$;
\item
\label{it:phi_1}
$\hat{u}_{ld}(\xi) = 0$ on a $C^{K-1}$ surface,
$\displaystyle \eta = \phi_1(\mu,\nu) = \nu \left(
-\frac{k_0 t_d}{\check{\delta} t_{(l+1)d}} \mu
- \frac{t_d}{t_{(l-1)d} t_{(l+1)d}} \nu + O(2) \right)$;
\item
\label{it:phi_2}
$\hat{u}_0(\xi) = 0$ on a $C^{K-1}$ surface,
$\displaystyle \nu = \phi_2(\mu,\eta) = \eta \left(
\frac{k_0 t_{(l-1)d}}{\check{\delta} t_{-d}} \mu
- \frac{t_{(l-1)d}}{t_d t_{-d}} \eta + O(2) \right)$;
\item
\label{it:Lambda}
there is a $C^{K-2}$ function given by
\begin{equation}
\Lambda(\xi) =
\left( \frac{k_0}{\check{\delta}} \mu +
\frac{1}{t_d} \eta + \frac{1}{t_{(l-1)d}} \nu \right)^2
- \frac{4 k_0}{\check{\delta} t_d} \mu \eta + o(2) \;,
\label{eq:Lambda}
\end{equation}
such that for $\mu > 0$ classical saddle-node bifurcations of $\cS$-cycles occur
when $\Lambda(\xi) = 0$ (though are not necessarily admissible);
\item
\label{it:zeta12}
if $\mu > 0$, then $\Lambda(\xi) \le 0$ only if $\eta \le 0$
and $\nu \ge \phi_2(\mu,\eta)$;
moreover $\Lambda(\xi) = 0$ along
$(\mu,0,\zeta_1(\mu))$ and 
$(\mu,\zeta_2(\mu),\phi_2(\mu,\zeta_2(\mu)))$ for
$C^{K-1}$ functions $\zeta_1$ and $\zeta_2$,
\begin{eqnarray}
\zeta_1(\mu) & = & -\frac{k_0 t_{(l-1)d}}{\check{\delta}} \mu +
O(\mu^2) \;, \label{eq:zeta1} \\
\zeta_2(\mu) & = & \frac{k_0 t_d}{\check{\delta}} \mu +
O(\mu^2) \;. \label{eq:zeta2}
\end{eqnarray}
\end{enumerate}
\label{le:expansions}
\end{lemma}

See Appendix \ref{sec:PROOFS} for a proof.
It is useful to consider the nonsingular, linear coordinate change:
\begin{equation}
\left[ \begin{array}{c}
\tilde{\mu} \\ \tilde{\eta} \\ \tilde{\nu}
\end{array} \right] =
\left[ \begin{array}{ccc}
\displaystyle{\frac{k_0}{\check{\delta}}}
 & -\displaystyle{\frac{1}{t_d}} & \displaystyle{\frac{1}{t_{(l-1)d}}} \\
0 & -\displaystyle{\frac{1}{t_d}} & -\displaystyle{\frac{1}{t_{(l-1)d}}} \\
0 & \displaystyle{\frac{1}{t_d}} & -\displaystyle{\frac{1}{t_{(l-1)d}}}
\end{array} \right]
\left[ \begin{array}{c}
\mu \\ \eta \\ \nu
\end{array} \right] \;,
\label{eq:xiTilde}
\end{equation}
because substitution of (\ref{eq:xiTilde}) into (\ref{eq:Lambda}) produces
\begin{equation}
\Lambda(\xi) = \tilde{\mu}^2 + \tilde{\eta}^2 - \tilde{\nu}^2 + o(2) \;.
\label{eq:LambdaAlt}
\end{equation}
In this alternative coordinate system,
saddle-node bifurcations occur approximately on a cone.
All six surfaces are sketched for $\mu \ge 0$ in Fig.~\ref{fig:gsp3d}.
The surface $\Lambda(\xi) = 0$ intersects the plane $\eta = 0$
tangentially along the curve $(\mu,0,\zeta_1(\mu))$
and the surface $\nu = \phi_2(\mu,\eta)$ tangentially along the curve
$(\mu,\zeta_2(\mu),\phi_2(\mu,\zeta_2(\mu)))$.

%%%%%%%%%%%%%%%%%%%%%%%%%%%%%%%%%%%%%%%%%%%%%%%%%%%%%%%%%%%%%
\begin{figure}[h!]
\begin{center}
\includegraphics[width=16cm,height=13cm]{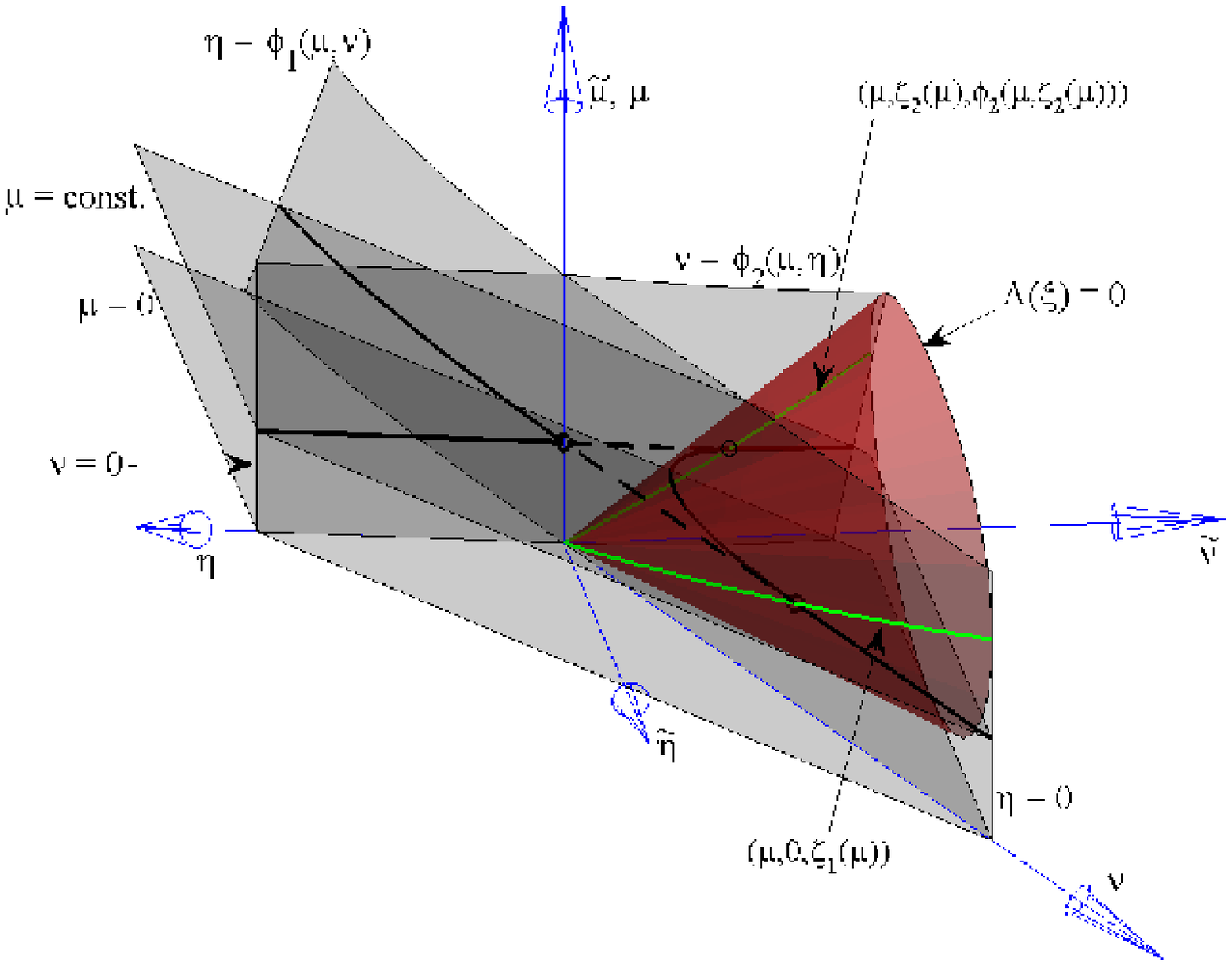}
%\setlength{\unitlength}{1cm}
%\begin{picture}(15,13)
%\put(-.5,0){\includegraphics[width=16cm,height=13cm]{gsp3d}}
%\end{picture}
\caption{
The three-dimensional unfolding of a generalized \shr~for the
\pws, continuous map (\ref{eq:pwsMap}).
The surface $\mu = 0$ corresponds to the \bcb~of a fixed point.
On each of the four surfaces, $\eta = 0$, $\nu = 0$,
$\eta = \phi_1(\mu,\nu)$ and $\nu = \phi_2(\mu,\eta)$
(which share a mutual intersection with the $\mu$-axis),
one point of an $\cS$-cycle lies on the \sw.
The surface $\Lambda(\xi) = 0$ corresponds to
classical saddle-node bifurcations of an $\cS$-cycle.
This surface is approximately a cone, as made evident by transformation to
the alternative $(\tilde{\mu},\tilde{\eta},\tilde{\nu})$-coordinate
system, (\ref{eq:xiTilde}), and tangentially intersects $\eta = 0$
and $\nu = \phi_2(\mu,\eta)$ along the curves
$(\mu,0,\zeta_1(\mu))$ and
$(\mu,\zeta_2(\mu),\phi_2(\mu,\zeta_2(\mu)))$, respectively.
The plane $\mu = {\rm const}$, illustrates the cross-section
shown in Fig.~\ref{fig:gsp2d}.
\label{fig:gsp3d}
}
\end{center}
\end{figure}
%%%%%%%%%%%%%%%%%%%%%%%%%%%%%%%%%%%%%%%%%%%%%%%%%%%%%%%%%%%%%

Different two-dimensional slices of parameter space through Fig.~\ref{fig:gsp3d}
will produce vastly different bifurcation sets.
Since slices defined by fixing the value of $\mu$
have been shown earlier
(see Fig.~\ref{fig:breakup}), in Fig.~\ref{fig:gsp3d}
we draw a plane at fixed $\mu > 0$ and show its curves
of intersection with the nearby surfaces.
This cross-section is shown again in Fig.~\ref{fig:gsp2d}.
A close inspection of the formulas for $\phi_1$ and $\phi_2$
given \Lemm~\ref{le:expansions} reveal a specific geometrical arrangement
near the generalized \shr~as follows.
Since $t_d < 0$, $t_{(l+1)d} > 0$
(\Lemm~\ref{le:shrPoint}(\ref{it:tneighbors})) and
${\rm sgn}(k_0) = {\rm sgn}(\check{\delta})$
(\ref{eq:muSign}), the coefficient for the $\mu \nu$ term
of $\phi_1(\mu,\nu)$ is positive.
Consequently, for small $\mu > 0$
the angle $\theta_2$ in Fig.~\ref{fig:gsp2d} is greater than $90^\circ$.
Similarly the coefficient for the $\mu \eta$ term of $\phi_2(\mu,\eta)$
is negative and so $\theta_1 < 90^\circ$.
Any smooth distortion of Fig.~\ref{fig:gsp2d} will
preserve the property $\theta_1 < \theta_2$.
The curves $\eta = \phi_1(\mu,\nu)$ and $\nu = \phi_2(\mu,\eta)$
bend left and down because the $\nu^2$ term of $\phi_1(\mu,\nu)$
and the $\eta^2$ term of $\phi_2(\mu,\eta)$ are both negative.
The saddle-node locus is approximately a parabola.

Finally, we present the main result.
See Appendix \ref{sec:PROOFS} for a proof.

\begin{theorem}
Suppose (\ref{eq:pwsMap}) is at a
generalized \shr~when $\xi = 0$ and that $K \ge 4$.
Assume that the only eigenvalue of $M_\cS(0)$  on the unit circle is $1$
and that it has algebraic multiplicity one.
Assume $y_0$ and $y_{ld}$ are the only points of $\{ y_i \}$ that
lie on the \sw.
Assume we have (\ref{eq:cu_0}), (\ref{eq:cu_ld}) and (\ref{eq:muSign}).
Let
\begin{eqnarray}
\Psi_1 & = & \{ \xi \in \mathbb{R}^3 ~|~ \mu > 0, \eta > 0, \nu > 0 \}
\;, \nonumber \\
\Psi_2 & = & \{ \xi \in \mathbb{R}^3 ~|~ \mu > 0, \eta < \phi_1(\mu,\nu),
\nu < \phi_2(\mu,\eta) \} \;, \nonumber \\
\Psi_3 & = & \{ \xi \in \mathbb{R}^3 ~|~ \mu > 0, \zeta_2(\mu) < \eta < 0,
\phi_2(\mu,\eta) < \nu < \zeta_1(\mu), \Lambda(\xi) > 0 \} \;, \nonumber
\end{eqnarray}
for the functions described in \Lemm~\ref{le:expansions}
(see Fig.~\ref{fig:gsp2d}).

Then, in $\Psi_1$, $\cS$ and $\check{\cS}$-cycles are admissible
and collide in \nsf~bifurcations on $\nu = 0$
and $\eta = 0$ for $\nu > \zeta_1(\mu)$,
in $\Psi_2$, $\cS$ and $\hat{\cS}$-cycles are admissible
and collide in \nsf~bifurcations on $\eta = \phi_1(\mu,\nu)$
and $\nu = \phi_2(\mu,\eta)$ for $\eta < \zeta_2(\mu)$,
in $\Psi_3$, two $\cS$-cycles are admissible and collide in classical saddle-node
bifurcations where $\Lambda(\xi) = 0$;
the boundaries $\eta = 0$ and $\nu = \phi_2(\mu,\eta)$
correspond to border-collision persistence.
\label{th:unfold}
\end{theorem}

%%%%%%%%%%%%%%%%%%%%%%%%%%%%%%%%%%%%%%%%%%%%%%%%%%%%%%%%%%%%%
\begin{figure}[ht]
\begin{center}
\includegraphics[width=9.6cm,height=8cm]{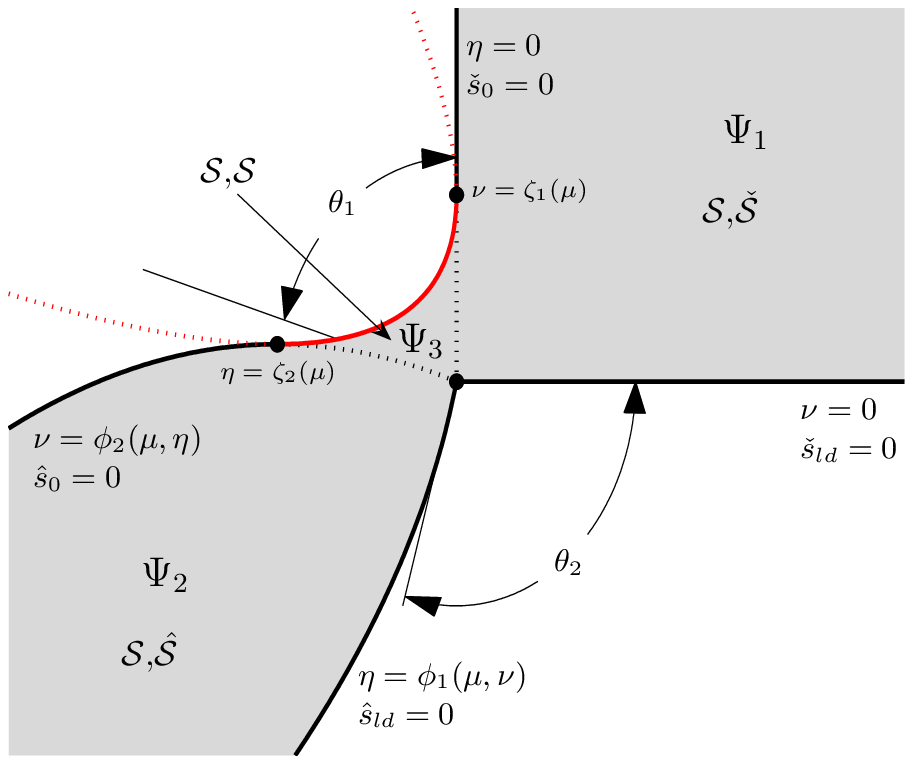}
%\setlength{\unitlength}{1cm}
%\begin{picture}(9.6,8)
%\put(3.7,5.8){\small $\theta_1$}
%\put(6,2.15){\small $\theta_2$}
%\put(8.5,3.7){\small $\nu = 0$}
%\put(5.1,7.4){\small $\eta = 0$}
%\put(.7,3.4){\small $\nu = \phi_2(\mu,\eta)$}
%\put(4,1){\small $\eta = \phi_1(\mu,\nu)$}
%\put(5.15,5.95){\scriptsize $\nu = \zeta_1(\mu)$}
%\put(2.6,4.1){\scriptsize $\eta = \zeta_2(\mu)$}
%\put(7.7,6.5){\large $\Psi_1$}
%\put(1.8,2){\large $\Psi_2$}
%\put(4.4,4.38){\large $\Psi_3$}
%\put(7.5,5.7){$\cS$,$\check{\cS}$}
%\put(1.6,1.2){$\cS$,$\hat{\cS}$}
%\put(2.4,6.1){$\cS$,$\cS$}
%\put(8.5,3.3){\small $\check{s}_{ld} = 0$}
%\put(5.1,7){\small $\check{s}_0 = 0$}
%\put(.7,3){\small $\hat{s}_0 = 0$}
%\put(4,.6){\small $\hat{s}_{ld} = 0$}
%\end{picture}
\caption{
A two-parameter bifurcation diagram for the map (\ref{eq:pwsMap})
near a generalized shrinking point for fixed, small $\mu > 0$.
As stated in \Ther~\ref{th:unfold}, an admissible $\cS$-cycle
coexists with,
(i) an admissible $\check{\cS}$-cycle in $\Psi_1$;
(ii) an admissible $\hat{\cS}$-cycle in $\Psi_2$;
(iii) a second admissible $\cS$-cycle in $\Psi_3$.
Along the solid curves 
the two coexisting, admissible, periodic solutions
undergo a \nsf~bifurcation
except on the curve
connecting the points $\nu = \zeta_1(\mu)$ and $\eta = \zeta_2(\mu)$
which corresponds to a classical saddle-node bifurcation of
the two periodic solutions.
The boundary between $\Psi_1$ and $\Psi_3$ and
the boundary between $\Psi_2$ and $\Psi_3$ correspond to
border-collision persistence.
\label{fig:gsp2d}
}
\end{center}
\end{figure}
%%%%%%%%%%%%%%%%%%%%%%%%%%%%%%%%%%%%%%%%%%%%%%%%%%%%%%%%%%%%%

%%%%%%%%%%%%%%%%%%%%%%%%%%
\section{Conclusions}
\label{sec:CONC}

We have studied resonance arising from an arbitrary
\bcb~of a \pws, continuous map.
When a periodic solution created in a \bcb~is
rotational, in the sense described in \S\ref{sec:SYMDYNS},
the corresponding resonance tongue typically 
has a sausage-like geometry, see Fig.~\ref{fig:rrEqrL_2}.
Shrinking points break apart as parameters are varied to
move away from the \bcb, Fig.~\ref{fig:breakup}.
We have proved \Ther~\ref{th:unfold} which
details the manner by which \shr~destruction occurs.
The results of the theorem are in complete agreement
with numerical results, Figs.~\ref{fig:manyBreak} and \ref{fig:breakup}.
Theorem \ref{th:unfold} does not provide
an understanding of global properties of \tong s,
for instance the observation that the majority of, or perhaps all of,
the kinks in the \tong s of Fig.~\ref{fig:manyBreak}
appear on the left sides of the tongues.
%This corresponds to $k_0$ (\ref{eq:k0}) being
%positive for every \shr~in the figure,
%but it is not clear why this should be so.

If the map (\ref{eq:pwsMap}) has an invariant circle
that intersects the \sw~at two points and the restriction of
the map to the circle is a homeomorphism,
then any periodic solution on the circle will have a
corresponding \sew~that is rotational.
Our results do not apply to periodic solutions born in
\bcb s that are non-rotational.
However such periodic solutions seem to be,
in some sense, less common \cite{SiMe08b}.
Note that we make no requirement that (\ref{eq:pwsMap}) is a homeomorphism
or has an invariant circle,
only that corresponding periodic solutions are rotational.

A limitation of \Ther~\ref{th:unfold} is that it
includes the assumption that the map (\ref{eq:pwsMap}) is piecewise-$C^K$.
Poincar\'{e} maps relating to sliding phenomena
are generically \pws, continuous
with a $\frac{3}{2}, 2, \frac{5}{2},\ldots$ type
power expansion \cite{DiKo02}
%(and an associated multiplier of $1$ on one side of the bifurcation)
and hence not apply to \Ther~\ref{th:unfold}.
It seems reasonable that in this case a similar result with
different scaling laws could apply.
%\footnote{
%In particular, the renormalized map will have a $+O(\mu^{\frac{1}{2}})$
%term instead of a $+O(\mu)$ term,
%suggesting that the points $A$ and $B$ in Fig.~\ref{fig:breakup}
%move away from $O$ like $\sqrt{\mu}$ instead of
%linearly with respect to $\mu$.
%The points $A$ and $B$ correspond to the coincidence of
%a \bcb~and a saddle-node bifurcation.
%In one dimension this scenario is unfolded by
%$x_{n+1} = \mu + (1+\eta) x_n + a x_n^2$
%for which saddle-node bifurcations occur
%along $\mu = \frac{\eta^2}{4 a}$.
%Replacing the quadratic term with a $3/2$-type term gives
%$x_{n+1} = \mu + (1+\eta) x_n + a x_n^{\frac{3}{2}}$
%for which saddle-node bifurcations occur
%along $\mu = \frac{4 \eta^3}{27 a^2}$.
%I.e.~the saddle-node locus is,
%to lowest order, cubic instead of quadratic.}
A major hurdle in the analysis is that the center manifold
theorem (applied in the proof of \Ther~\ref{th:unfold})
is not immediately applicable to non-integer power expansions.

In our definition of the codimension-three generalized
\shr~(\Defi~\ref{def:gsp}),
we require that both matrices 
$I-M_{\check{\cS}}$ and $I-M_{\hat{\cS}}$ are nonsingular at $\xi = 0$.
We then prove \Lemm~\ref{le:expansions}
by computing parameter-dependent series expansions
(see Appendix \ref{sec:PROOFS}).
%\footnote{
%The proof of
%Lemma \ref{le:expansions} doesn't use the fact
%that $\{ y_i \}$ is admissible (except in applying Lemma \ref{le:shrPoint}).
%The proof of Lemma \ref{le:shrPoint} (given in \cite{SiMe09})
%relies heavily on the fact that $\{ y_i \}$ is admissible.
%However I have since discovered how to prove that
%$y_d \ne y_0$ without applying admissibility
%(the proof requires showing that equality would
%lead to $y_{d+k} = y_k$, $\forall k$
%by induction and cases)
%Part (\ref{it:yiperiodn}) follows since  ${\rm gcd}(d,n) = 1$.
%Without admissibility only the signs in (\ref{it:tneighbors})
%remain undetermined.}.
However part (\ref{it:ddFormula}) of \Lemm~\ref{le:expansions}
is independent of the parameters
and it seems there should be a more direct proof of this result.
This problem remains for future work.
Related problems that remain to be fully understood include
the persistence of invariant topological circles created at \bcb s and
the simultaneous occurrence of a \bcb~with a
classical Neimark-Sacker bifurcation.

%%%%%%%%%%%%%%%%%%%%%%%%%%
\appendix

\section{Proofs for Section \ref{sec:UNFOLD}}
\label{sec:PROOFS}

\begin{proof}[Proof of \Lemm~\ref{le:expansions}]
We divide the proof into four major steps.
In the first step we use the formulas for
$\check{u}_0(\xi)$ (\ref{eq:cu_0}) and
$\check{u}_{ld}(\xi)$ (\ref{eq:cu_ld}) to
derive formulas for $\hat{u}_0$ and $\hat{u}_{ld}$,
enabling us to compute the border-collision boundaries
$\eta = \phi_1(\mu,\nu)$ and
$\nu = \phi_2(\mu,\eta)$.
The method used is an extension of \Lemm~3 of \cite{SiMe09} to
account for the nonlinear terms in (\ref{eq:pwsMap}).
In step 2 we continue this methodology to derive identities
relating coefficients, verify (\ref{it:ddFormula}) of the lemma
and refine the formulas for $\phi_1$ and $\phi_2$.
In the third step we compute the 
one-dimensional center manifold of $f^\cS$
to identify saddle-node bifurcations on the surface $\Lambda = 0$.
Lastly, in step 4 we study at the intersections
of $\Lambda = 0$ with $\eta = 0$ and $\nu = \phi_2(\mu,\eta)$
to complete the proof.
For convenience we write
\begin{equation}
\det(I-M_{\cS}(0,\eta,\nu)) = k_1 \eta + k_2 \nu + O(2) \;.
\label{eq:k1k2}
\end{equation}

{\bf Step 1:} {\em Derive formulas for $\hat{u}_0$ and $\hat{u}_{ld}$
and define $\phi_1$ and $\phi_2$.}\\
Periodic solutions of (\ref{eq:pwsMap}) with 
associated symbol sequences that differ by a single symbol
may be related algebraically (as shown below).
However the sequences $\check{\cS}$ and $\hat{\cS}$ differ in two symbols,
so in general we cannot effectively compare the
$\check{\cS}$ and $\hat{\cS}$-cycles.
But if one of the four important quantities,
$\check{u}_0$, $\check{u}_{ld}$, $\hat{u}_0$ and $\hat{u}_{ld}$ is zero,
then one of the two $n$-cycles is also an $\cS$-cycle
and the two $n$-cycles may indeed be effectively related.
The four curves which make up the edges of a
\tong~near a \shr~correspond
to where these four quantities are zero.
We separate this step of the proof by where
different assumptions on the parameters are made.

{\bf a)} {\em Suppose $\eta = 0$.}
Then by (\ref{eq:cu_0}), $\check{u}_0 = 0$.
Thus the $\check{\cS}$-cycle, $\{ \check{x}_i \}$, is also an $\cS$-cycle.
So each $\check{x}_i$ is a fixed point of $f^{\cS^{(i)}}$.
In particular, $\check{x}_{ld}$ is a fixed point of $f^{\cS^{(ld)}}$,
which in the renormalized frame (\ref{eq:renormMap})
says that $\check{z}_{ld}$ is a fixed point of $h^{\cS^{(ld)}}$:
\begin{equation}
\check{z}_{ld}(\mu,0,\nu) =
h^{\cS^{(ld)}}(\check{z}_{ld}(\mu,0,\nu);\mu,0,\nu) \;.
\nonumber
\end{equation}
Expressing $h^{\cS^{(ld)}}$ as a Taylor series centered at $y_{ld}$
(\ref{eq:yidef}) and evaluated at $\check{z}_{ld}$ yields
\begin{equation}
\check{z}_{ld}(\mu,0,\nu) = h^{\cS^{(ld)}}(y_{ld};\mu,0,\nu)
+ D_z h^{\cS^{(ld)}}(y_{ld};\mu,0,\nu)
\big( \check{z}_{ld}(\mu,0,\nu) - y_{ld} \big) + O(3) \;,
\label{eq:exp_1a}
\end{equation}
where the next term in the expansion is
$O(\mu) O(|\check{z}_{ld}(\mu,0,\nu) - y_{ld}|^2)$ which is $O(3)$
because $\check{z}_{ld}(0,0,0) = y_{ld}$.

Now, $\hat{z}_{ld}$ is a fixed point of $h^{\hat{\cS}^{(ld)}}$
(for any sufficiently small $\xi$), therefore
\begin{equation}
\hat{z}_{ld}(\mu,0,\nu) = h^{\hat{\cS}^{(ld)}}(y_{ld};\mu,0,\nu)
+ D_z h^{\hat{\cS}^{(ld)}}(y_{ld};\mu,0,\nu)
\big( \hat{z}_{ld}(\mu,0,\nu) - y_{ld} \big) + O(3) \;,
\label{eq:exp_1b}
\end{equation}
where the right-hand side is the Taylor series
of $h^{\hat{\cS}^{(ld)}}$ centered at $y_{ld}$.
But $\cS^{(ld)} = \hat{\cS}^{(ld) \overline{0}}$
(refer to \S\ref{sec:SYMDYNS}),
thus whenever $u = e_1^{\sf T} z = 0$, we have
$h^{\cS^{(ld)}}(z;\xi) = h^{\hat{\cS}^{(ld)}}(z;\xi)$
(by continuity of (\ref{eq:pwsMap})).
Since $t_{ld} = e_1^{\sf T} y_{ld} = 0$
(\Lemm~\ref{le:shrPoint}(\ref{it:t0tld})),
we have
\begin{equation}
h^{\cS^{(ld)}}(y_{ld};\mu,0,\nu) = h^{\hat{\cS}^{(ld)}}(y_{ld};\mu,0,\nu) \;.
\label{eq:exp_1c}
\end{equation}
Furthermore, although in general the last $N-1$ columns of
$D_z h^{\cS^{(ld)}}$ and $D_z h^{\hat{\cS}^{(ld)}}$ are not equal,
since $h^{\cS^{(ld)}} \equiv h^{\hat{\cS}^{(ld)}}$ on the \sw,
the last $N-1$ columns of
$D_z h^{\cS^{(ld)}}(y_{ld};\mu,0,\nu)$ and
$D_z h^{\hat{\cS}^{(ld)}}(y_{ld};\mu,0,\nu)$ are indeed equal.
Consequently the first row of their adjugates is the same:
\begin{equation}
e_1^{\sf T} {\rm adj}(I - D_z h^{\cS^{(ld)}}(y_{ld};\mu,0,\nu)) =
e_1^{\sf T} {\rm adj}(I - D_z h^{\hat{\cS}^{(ld)}}(y_{ld};\mu,0,\nu)) \;.
\label{eq:exp_rhoT_1}
\end{equation}
The combination of (\ref{eq:exp_1a}), (\ref{eq:exp_1b}),
(\ref{eq:exp_1c}) and $t_{ld} = 0$ produces
\begin{equation}
\big( I - D_z h^{\cS^{(ld)}}(y_{ld};\mu,0,\nu) \big) \check{z}_{ld}(\mu,0,\nu) =
\big( I - D_z h^{\hat{\cS}^{(ld)}}(y_{ld};\mu,0,\nu) \big)
\hat{z}_{ld}(\mu,0,\nu) + O(3) \;,
\label{eq:exp_comb_1}
\end{equation}
and multiplication of both sides of (\ref{eq:exp_comb_1}) by (\ref{eq:exp_rhoT_1})
on the left (remembering (\ref{eq:adjRelation}) and
$u_i = e_1^{\sf T} z_i$) leads to
\begin{equation}
\det \big( I - D_z h^{\cS^{(ld)}}(y_{ld};\mu,0,\nu) \big) \check{u}_{ld}(\mu,0,\nu) =
\det \big( I - D_z h^{\hat{\cS}^{(ld)}}(y_{ld};\mu,0,\nu) \big) \hat{u}_{ld}(\mu,0,\nu) +
O(3) \;.
\nonumber
\end{equation}
We now plug in known expansions
($\check{u}_{ld}$ is given by (\ref{eq:cu_ld}),
$\det(I - D_z h^{\hat{\cS}^{(ld)}}(y_{ld};\mu,0,\nu)) = \hat{\delta} + O(1)$
by (\ref{eq:dHat}) and \Lemm~\ref{le:cyclicDetDx},
$\det(I - D_z h^{\cS^{(ld)}}(y_{ld};\mu,0,\nu)) = k_0 \mu + k_2 \nu + O(2)$
by (\ref{eq:k0}), (\ref{eq:k1k2}) and \Lemm~\ref{le:cyclicDetDx}):
\begin{equation}
(k_0 \mu + k_2 \nu + O(2)) \nu (1 + O(1)) =
(\hat{\delta} + O(1)) \hat{u}_{ld}(\mu,0,\nu) \;,
\nonumber
\end{equation}
which, upon rearranging, produces the following useful expression:
\begin{equation}
\hat{u}_{ld}(\mu,0,\nu) = \frac{k_0}{\hat{\delta}} \mu \nu +
\frac{k_2}{\hat{\delta}} \nu^2 + O(3) \;.
\label{eq:exp_res_1}
\end{equation}
Below we use the same approach to derive
further expressions for $\hat{u}_0$ and $\hat{u}_{ld}$
with different assumptions on the parameters.
For brevity we will not provide the same level of detail.

{\bf b)} {\em Suppose $\eta = 0$ and $\mu = 0$.}
As above, since $\eta = 0$, the $\check{\cS}$-cycle is also an $\cS$-cycle.
In particular, $\check{z}_{-d}$ is a fixed point of $h^{\cS^{(-d)}}$.
Here we express $h^{\cS^{(-d)}}$ as a Taylor series centered at $y_0$
(the reason for choosing $y_0$ will soon become clear)
and evaluated at $\check{z}_{-d}$:
\begin{equation}
\check{z}_{-d}(0,0,\nu) = h^{\cS^{(-d)}}(y_0;0,0,\nu)
+ D_z h^{\cS^{(-d)}}(y_0;0,0,\nu)
\big( \check{z}_{-d}(0,0,\nu) - y_0 \big) \;,
\label{eq:exp_2a}
\end{equation}
where, unlike in (\ref{eq:exp_1a}),
there is no next term in the expansion when $\mu = 0$
the map is affine.
Now, $\hat{z}_0$ is a fixed point of $h^{\hat{\cS}}$, therefore
\begin{equation}
\hat{z}_0(0,0,\nu) = h^{\hat{\cS}}(y_0;0,0,\nu)
+ D_z h^{\hat{\cS}}(y_0;0,0,\nu)
\big( \hat{z}_0(0,0,\nu) - y_0 \big) \;.
\label{eq:exp_2b}
\end{equation}
But $\cS^{(-d)} = \hat{\cS}^{\overline{0}}$ (by (\ref{eq:rss})),
and since $t_0 = 0$ (\Lemm~\ref{le:shrPoint}(\ref{it:t0tld}))
we may perform the same simplification that we did above to combine
(\ref{eq:exp_2a}) and (\ref{eq:exp_2b}) leaving
\begin{eqnarray}
\big( I - D_z h^{\cS^{(-d)}}(y_0;0,0,\nu) \big) \check{z}_{-d}(0,0,\nu) & = &
\big( I - D_z h^{\hat{\cS}}(y_0;0,0,\nu) \big) \hat{z}_0(0,0,\nu) \;, \nonumber \\
\Rightarrow~~~~
\det \big( I - D_z h^{\cS^{(-d)}}(y_0;0,0,\nu) \big) \check{u}_{-d}(0,0,\nu) & = &
\det \big( I - D_z h^{\hat{\cS}}(y_0;0,0,\nu) \big) \hat{u}_0(0,0,\nu)
\;, \nonumber \\
\Rightarrow~~~~
(k_2 \nu + O(\nu^2))(t_{-d} + O(\nu)) & = &
(\hat{\delta} + O(\nu)) \hat{u}_0(0,0,\nu) \;, \nonumber \\
\Rightarrow~~~~
\hat{u}_0(0,0,\nu) & = &
\frac{k_2 t_{-d}}{\hat{\delta}} \nu + O(\nu^2) \;. \label{eq:exp_res_2}
\end{eqnarray}

{\bf c)} {\em Suppose $\nu = 0$.}
As discussed in \S\ref{sec:UNFOLD},
when $\nu = 0$ the $\check{\cS}$-cycle is also an $\cS^{(-d)}$-cycle.
Therefore $\check{z}_0$ is a fixed point of $h^{\cS^{(-d)}}$ and so
\begin{equation}
\check{z}_0(\mu,\eta,0) = h^{\cS^{(-d)}}(y_0;\mu,\eta,0)
+ D_z h^{\cS^{(-d)}}(y_0;\mu,\eta,0)
\big( \check{z}_0(\mu,\eta,0) - y_0 \big) + O(3) \;.
\label{eq:exp_3a}
\end{equation}
Also, $\hat{z}_0$ is a fixed point of $h^{\hat{\cS}}$, thus
\begin{equation}
\hat{z}_0(\mu,\eta,0) = h^{\hat{\cS}}(y_0;\mu,\eta,0)
+ D_z h^{\hat{\cS}}(y_0;\mu,\eta,0)
\big( \hat{z}_0(\mu,\eta,0) - y_0 \big) + O(3) \;.
\label{eq:exp_3b}
\end{equation}
Combining (\ref{eq:exp_3a}) and (\ref{eq:exp_3b}) leads to
(since $\cS^{(-d)} = \hat{\cS}^{\overline{0}}$)
\begin{equation}
\det \big( I - D_z h^{\cS^{(-d)}}(y_0;\mu,\eta,0) \big) \check{u}_0(\mu,\eta,0) =
\det \big( I - D_z h^{\hat{\cS}}(y_0;\mu,\eta,0) \big) \hat{u}_0(\mu,\eta,0) +
O(3) \;.
\nonumber
\end{equation}
Take care to note that
\begin{equation}
\tilde{k_0} \equiv \frac{\partial}{\partial \mu}
\det(I - D_x f^\cS(\check{x}_d(\xi);\xi)) \Big|_{\xi = 0} \;,
\label{eq:m2}
\end{equation}
is different to (\ref{eq:m1}) due to the presence of nonlinear terms
in (\ref{eq:pwsMap}) (below we will show that in fact $\tilde{k_0} = -k_0$).
Consequently
\begin{eqnarray}
(\tilde{k_0} \mu + k_1 \eta + O(2)) \eta (1 + O(1)) & = &
(\hat{\delta} + O(1)) \hat{u}_0(\mu,\eta,0) \;, \nonumber \\
\Rightarrow~~~~
\hat{u}_0(\mu,\eta,0) & = & \frac{\tilde{k_0}}{\hat{\delta}} \mu \eta +
\frac{k_1}{\hat{\delta}} \eta^2 + O(3) \;. \label{eq:exp_res_3}
\end{eqnarray}

{\bf d)} {\em Suppose $\nu = 0$ and $\mu = 0$.}
Here $\check{z}_{(l+1)d}$ is a fixed point of $h^{\cS^{(ld)}}$, so
\begin{equation}
\check{z}_{(l+1)d}(0,\eta,0) = h^{\cS^{(ld)}}(y_{ld};0,\eta,0)
+ D_z h^{\cS^{(ld)}}(y_{ld};0,\eta,0)
\big( \check{z}_{(l+1)d}(0,\eta,0) - y_{ld} \big) \;,
\label{eq:exp_4a}
\end{equation}
and $\hat{z}_{ld}$ is a fixed point of $h^{\hat{\cS}^{(ld)}}$, so
\begin{equation}
\hat{z}_{ld}(0,\eta,0) = h^{\hat{\cS}^{(ld)}}(y_{ld};0,\eta,0)
+ D_z h^{\hat{\cS}^{(ld)}}(y_{ld};0,\eta,0)
\big( \hat{z}_{ld}(0,\eta,0) - y_{ld} \big) \;,
\label{eq:exp_4b}
\end{equation}
and since $\cS^{(ld)} = \hat{\cS}^{(ld) \overline{0}}$ we obtain
\begin{eqnarray}
\det \big( I - D_z h^{\cS^{(ld)}}(y_{ld};0,\eta,0) \big) \check{u}_{(l+1)d}(0,\eta,0) & = &
\det \big( I - D_z h^{\hat{\cS}^{(ld)}}(y_{ld};0,\eta,0) \big) \hat{u}_{ld}(0,\eta,0) \;, \nonumber \\
\Rightarrow~~~~
(k_1 \eta + O(\eta^2))(t_{(l+1)d} + O(\eta)) & = &
(\hat{\delta} + O(\eta)) \hat{u}_{ld}(0,\eta,0) \;, \nonumber \\
\Rightarrow~~~~
\hat{u}_{ld}(0,\eta,0) & = &
\frac{k_1 t_{(l+1)d}}{\hat{\delta}} \eta + O(\eta^2) \;. \label{eq:exp_res_4}
\end{eqnarray}

We now apply the implicit function theorem to the $C^{K-1}$
function $\hat{u}_{ld}(\xi)$.
By (\ref{eq:exp_res_1}) and (\ref{eq:exp_res_4}),
there exists a unique $C^{K-1}$ function $\phi_1$
such that for small $\mu$ and $\nu$,
$\hat{u}_{ld}(\mu,\phi_1(\mu,\nu),\nu) = 0$ and
\begin{equation}
\phi_1(\mu,\nu) = -\frac{k_0}{k_1 t_{(l+1)d}} \mu \nu
- \frac{k_2}{k_1 t_{(l+1)d}} \nu^2 + O(3) \;.
\label{eq:phi_1_temp}
\end{equation}
Recall that $\check{u}_{ld} = 0$ along the $\mu$-axis
(a consequence of (\ref{eq:cu_0}) and (\ref{eq:cu_ld})),
therefore $\phi_1 = 0$ whenever $\nu = 0$.
Thus we may rewrite (\ref{eq:phi_1_temp}) as
\begin{equation}
\phi_1(\mu,\nu) = \nu \left( -\frac{k_0}{k_1 t_{(l+1)d}} \mu
- \frac{k_2}{k_1 t_{(l+1)d}} \nu + O(2) \right) \;.
\label{eq:phi_1}
\end{equation}
Similarly by (\ref{eq:exp_res_2}) and (\ref{eq:exp_res_3}),
there exists a unique $C^{K-1}$ function $\phi_2$
such that for small $\mu$ and $\eta$,
$\hat{u}_0(\mu,\eta,\phi_2(\mu,\eta)) = 0$ and
\begin{equation}
\phi_2(\mu,\eta) = \eta \left( -\frac{\tilde{k_0}}{k_2 t_{-d}} \mu
- \frac{k_1}{k_2 t_{-d}} \eta + O(2) \right) \;,
\label{eq:phi_2}
\end{equation}
where the $\eta$ may be factored in the same fashion as for (\ref{eq:phi_1}).

{\bf Step 2:} {\em Verify (\ref{it:ddFormula})
of the lemma and refine the formulas for $\phi_1$ and $\phi_2$,
(\ref{eq:phi_1}) and (\ref{eq:phi_2}).}\\
Here we continue to employ the methodology above
to compare the $\check{\cS}$ and $\hat{\cS}$-cycles.

{\bf a)} {\em Suppose $\eta = \phi_1(\mu,\nu)$.}
Then $\hat{u}_{ld} = 0$,
and so the $\hat{\cS}$-cycle is also an $\cS$-cycle.
Thus, in particular, $\hat{z}_0$ is a fixed point of $h^\cS$:
\begin{eqnarray}
\hat{z}_0(\mu,\phi_1(\mu,\nu),\nu)) & = &
h^\cS(y_0;\mu,\phi_1(\mu,\nu),\nu) \nonumber \\
& & +~D_z h^\cS(y_0;\mu,\phi_1(\mu,\nu),\nu)
\big( \hat{z}_0(\mu,\phi_1(\mu,\nu),\nu) - y_0 \big)
+ O(3) \;. \nonumber
%\label{eq:exp_5a}
\end{eqnarray}
Also $\check{z}_0$ is a fixed point of $h^{\check{\cS}}$:
\begin{eqnarray}
\check{z}_0(\mu,\phi_1(\mu,\nu),\nu)) & = &
h^{\check{\cS}}(y_0;\mu,\phi_1(\mu,\nu),\nu) \nonumber \\
& & +~D_z h^{\check{\cS}}(y_0;\mu,\phi_1(\mu,\nu),\nu)
\big( \check{z}_0(\mu,\phi_1(\mu,\nu),\nu) - y_0 \big) + O(3) \;. \nonumber
%\label{eq:exp_5b}
\end{eqnarray}
Then, since $\cS = \check{\cS}^{\overline{0}}$,
\begin{eqnarray}
& \det \big( I - D_z h^\cS(y_0;\mu,\phi_1(\mu,\nu),\nu)) \big)
\hat{u}_0(\mu,\phi_1(\mu,\nu),\nu)) = & \nonumber \\
& \det \big( I - D_z h^{\check{\cS}}(y_0;\mu,\phi_1(\mu,\nu),\nu)) \big)
\check{u}_0(\mu,\phi_1(\mu,\nu),\nu)) + O(3) \;. &
\label{eq:exp_5det}
\end{eqnarray}
Unlike for similar expressions in Step 1,
we have already determined an expansion for each of the four components
of (\ref{eq:exp_5det})
(in particular $\hat{u}_0(\mu,\phi_1(\mu,\nu),\nu)$
is given by (\ref{eq:exp_res_2}), (\ref{eq:exp_res_3}) and (\ref{eq:phi_1}) and
$\check{u}_0(\mu,\phi_1(\mu,\nu),\nu)$
is given by (\ref{eq:cu_0}) and (\ref{eq:phi_1}),
also recall \Lemm~\ref{le:cyclicDetDx}):
\begin{eqnarray}
& (k_0 \mu + k_2 \nu + O(2)) \left( \frac{k_2 t_{-d}}{\hat{\delta}} \nu +
O(2) \right) = & \nonumber \\
& (\check{\delta} + O(1)) \left( -\frac{k_0}{k_1 t_{(l+1)d}} \mu \nu
- \frac{k_2}{k_1 t_{(l+1)d}} \nu^2 + O(3) \right) + O(3) \;. & \nonumber
%\Rightarrow~~~~
%\frac{k_0 k_2 t_{-d}}{\hat{\delta}} \mu \nu +
%\frac{k_2^2 t_{-d}}{\hat{\delta}} \nu^2 & = &
%-\frac{k_0 \check{\delta}}{k_1 t_{(l+1)d}} \mu \nu
%-\frac{k_2 \check{\delta}}{k_1 t_{(l+1)d}} \nu^2 + O(3) \;. \nonumber
\end{eqnarray}
Equating the second-order coefficients produces
\begin{equation}
k_1 k_2 t_{(l+1)d} t_{-d} = -\check{\delta} \hat{\delta} \;.
\label{eq:exp_res_5}
\end{equation}
%where the $k_0$ in each term has been cancelled\footnote{
%I can do this because by assumption $k_0 \ne 0$.}.
%Equating the two $\nu^2$ terms also produces (\ref{eq:exp_res_5})\footnote{
%This seems odd. Is there a reason for this?}.

{\bf b)} {\em Suppose $\eta = \phi_1(\mu,\nu)$ and $\mu = 0$.}
In the same manner as above,
equating the first components of the power series of 
$h^{\cS^{((l-1)d)}}$ and $h^{\check{\cS}^{(ld)}}$ yields
%Then $\hat{z}_{(l-1)d}$ is a fixed point of $h^{\cS^{((l-1)d)}}$, thus
%\begin{eqnarray}
%\hat{z}_{(l-1)d}(0,\phi_1(0,\nu),\nu)) & = &
%h^{\cS^{((l-1)d)}}(y_{ld};0,\phi_1(0,\nu),\nu) \nonumber \\
%& & +~D_z h^{\cS^{((l-1)d)}}(y_{ld};0,\phi_1(0,\nu),\nu)
%\big( \hat{z}_{(l-1)d}(0,\phi_1(0,\nu),\nu) - y_{ld} \big) \;,
%\label{eq:exp_6a}
%\end{eqnarray}
%and, in particular, $\check{z}_{ld}$ is a fixed point of $h^{\check{\cS}^{(ld)}}$, hence
%\begin{eqnarray}
%\check{z}_{ld}(0,\phi_1(0,\nu),\nu)) & = &
%h^{\check{\cS}^{(ld)}}(y_{ld};0,\phi_1(0,\nu),\nu) \nonumber \\
%& & +~D_z h^{\check{\cS}^{(ld)}}(y_{ld};0,\phi_1(0,\nu),\nu)
%\big( \check{z}_{ld}(0,\phi_1(0,\nu),\nu) - y_{ld} \big) \;,
%\label{eq:exp_6b}
%\end{eqnarray}
\begin{eqnarray}
%& \Rightarrow~~~~
& \det(I - D_z h^{\cS^{((l-1)d)}}(y_{ld};0,\phi_1(0,\nu),\nu)
\hat{u}_{(l-1)d}(0,\phi_1(0,\nu),\nu) = & \nonumber \\
& \det(I - D_z h^{\check{\cS}^{(ld)}}(y_{ld};0,\phi_1(0,\nu),\nu)
\check{u}_{ld}(0,\phi_1(0,\nu),\nu) \;, & \nonumber
\end{eqnarray}
and since $\cS^{((l-1)d)} = \check{\cS}^{(ld) \overline{0}}$,
\begin{eqnarray}
(k_2 \nu + O(\nu^2))(t_{(l-1)d} + O(\nu)) & = &
(\check{\delta} + O(\nu))(\nu + O(\nu^2)) \;, \nonumber \\
%\Rightarrow~~~~
%k_2 t_{(l-1)d} \nu & = & \check{\delta} \nu + O(\nu^2) \;, \nonumber \\
\Rightarrow~~~~
k_2 & = & \frac{\check{\delta}}{t_{(l-1)d}} \;. \label{eq:exp_res_6}
\end{eqnarray}

{\bf c)} {\em Suppose $\nu = \phi_2(\mu,\eta)$ and $\mu = 0$.}
Here, equating the first components of $h^\cS$ and $h^{\check{\cS}}$ produces
%Here $\hat{u}_0 = 0$
%thus $\hat{z}_d$ is a fixed point of $h^\cS$, i.e.
%\begin{eqnarray}
%\hat{z}_d(0,\eta,\phi_2(0,\eta)) & = &
%h^\cS(y_0;0,\eta,\phi_2(0,\eta)) \nonumber \\
%& & +~D_z h^\cS(y_0;0,\eta,\phi_2(0,\eta))
%\big( \hat{z}_d(0,\eta,\phi_2(0,\eta)) - y_0 \big) \;,
%\label{eq:exp_7a}
%\end{eqnarray}
%and $\check{z}_0$ is a fixed point of $h^{\check{\cS}}$, i.e.
%\begin{eqnarray}
%\check{z}_0(0,\eta,\phi_2(0,\eta)) & = &
%h^{\check{\cS}}(y_0;0,\eta,\phi_2(0,\eta)) \nonumber \\
%& & +~D_z h^{\check{\cS}}(y_0;0,\eta,\phi_2(0,\eta))
%\big( \check{z}_0(0,\eta,\phi_2(0,\eta)) - y_0 \big) \;,
%\label{eq:exp_7b}
%\end{eqnarray}
\begin{eqnarray}
%& \Rightarrow~~~~
& \det(I - D_z h^\cS(y_0;0,\eta,\phi_2(0,\eta))
\hat{u}_d(0,\eta,\phi_2(0,\eta)) = & \nonumber \\
& \det(I - D_z h^{\check{\cS}}(y_0;0,\eta,\phi_2(0,\eta))
\check{u}_0(0,\eta,\phi_2(0,\eta)) \;, & \nonumber
\end{eqnarray}
from which it follows that
\begin{eqnarray}
%\Rightarrow~~~~
(k_1 \eta + O(\eta^2))(t_d + O(\eta)) & = &
(\check{\delta} + O(\eta))(\eta + O(\eta^2)) \;, \nonumber \\
%\Rightarrow~~~~
%k_1 t_d \eta & = & \check{\delta} \eta + O(\eta^2) \;, \nonumber \\
\Rightarrow~~~~
k_1 & = & \frac{\check{\delta}}{t_d} \;. \label{eq:exp_res_7}
\end{eqnarray}

We now combine above equations to demonstrate
some parts of the lemma.
By combining (\ref{eq:exp_res_5}), (\ref{eq:exp_res_6}) and
(\ref{eq:exp_res_7}) we obtain (\ref{it:ddFormula}).
Combining (\ref{eq:phi_1}), (\ref{eq:exp_res_6}) and (\ref{eq:exp_res_7})
verifies (\ref{it:phi_1}).
Combining (\ref{eq:phi_2}), (\ref{eq:exp_res_6}) and (\ref{eq:exp_res_7})
will verify (\ref{it:phi_2}) once
it is shown that $\tilde{k_0} = -k_0$
%(\ref{eq:k0relation})
(see below).

{\bf Step 3:} {\em Derive and analyze the one-dimensional center manifold
of $f^\cS$ to obtain the function $\Lambda(\xi)$ and
identify saddle-node bifurcations of $\cS$-cycles.}\\
%Consider the map $f^\cS(x;\xi)$.
When $\xi \equiv (\mu,\eta,\nu) = 0$, $x = 0$ is a fixed point of $f^\cS(x;\xi)$
and the associated stability multipliers are
the eigenvalues of $D_x f^\cS(0;0) = M_\cS(0)$.
The matrix $M_\cS(0)$ has an eigenvalue 1
of algebraic multiplicity one.
Let $v \in \mathbb{R}^N$ be the associated eigenvector, i.e. $M_\cS(0) v = v$.
Notice $v \ne 0$ implies $e_1^{\sf T} v \ne 0$ since if not then
$M_{\cS^{\overline{0}}}(0) v = v$ which contradicts the assumption that
$I - M_{\check{\cS}}(0)$ is nonsingular (\Defi~\ref{def:gsp}).
In what follows we assume $e_1^{\sf T} v = 1$.

We now compute the restriction of $f^\cS$
to the one-dimensional center manifold.
Let
\begin{equation}
F(x;\xi) = \left[ \begin{array}{c}
f^\cS(x;\xi) \\
\xi
\end{array} \right] \;,
\nonumber
\end{equation}
denote the $(N+3)$-dimensional, $C^K$, extended map.
The Jacobian,
\begin{equation}
D F(0;0) = \left[ \begin{array}{ccccc}
M_\cS(0) & \Bigg| & P_\cS(0) b(0) & 0 & 0 \\
\hline
0 & \Bigg| & & I &
\end{array} \right] \;,
%\label{eq:DFproof}
\end{equation}
has a four-dimensional centerspace, $E^c \in \mathbb{R}^{N+3}$, spanned by
\begin{equation}
\left\{
\left[ \begin{array}{c} v \\ \hline 0 \\ 0 \\ 0 \end{array} \right] \;,
\left[ \begin{array}{c} y_0 \\ \hline 1 \\ 0 \\ 0 \end{array} \right] \;,
\left[ \begin{array}{c} 0 \\ \hline 0 \\ 1 \\ 0 \end{array} \right] \;,
\left[ \begin{array}{c} 0 \\ \hline 0 \\ 0 \\ 1 \end{array} \right]
\right\} \;, \nonumber
\end{equation}
since, in particular,
$y_0 = h^\cS(y_0;0) = M_\cS(0) y_0 + P_\cS(0) b(0)$.

Since $e_1^{\sf T} v = 1$, we may use the center manifold theorem
to express the local center manifold, $W^c$, of $f^\cS$,
in terms of $s = e_1^{\sf T} x$ and $\xi$.
In particular, on $W^c$,
\begin{equation}
x = X(s;\xi) = s v + \mu y_0 + O(2) \;,
\label{eq:Hproof}
\end{equation}
where $X : \mathbb{R} \times \mathbb{R}^3 \to \mathbb{R}^N$ is $C^{K-1}$.
The first component of the restriction of $f^\cS$ to $W^c$ is given by
\begin{eqnarray}
s'(s;\xi) & = & e_1^{\sf T} f^\cS(X(s;\xi);\xi) \nonumber \\
& = & e_1^{\sf T} \mu P_\cS(0) b(0) +
e_1^{\sf T} M_\cS(0) ( s v + \mu y_0 ) + O(2) \nonumber \\
& = & e_1^{\sf T} (P_\cS(0) b(0) + M_\cS(0) y_0) \mu +
e_1^{\sf T} M_\cS(0) v s + O(2) \nonumber \\
& = & s + O(2) \;, \nonumber
\end{eqnarray}
since $e_1^{\sf T} (P_\cS(0) b(0) + M_\cS(0) y_0) = e_1^{\sf T} y_0 = t_0 = 0$
and $e_1^{\sf T} M_\cS(0) v = e_1^{\sf T} v = 1$.
However we require the knowledge of second order terms of
$s'(s;\xi)$, so write
\begin{equation}
s'(s;\xi) = s + c_1 s^2 + c_2 \mu s + c_3 \eta s + c_4 \nu s
+ c_5 \mu^2 + c_6 \mu \eta + c_7 \mu \nu
+ c_8 \eta^2 + c_9 \eta \nu + c_{10} \nu^2 + O(3) \;.
\label{eq:sPrime1}
\end{equation}
We now utilize known properties of $f^\cS$ to determine
an expression for the majority of the
coefficients, $c_i$.
\begin{enumerate}
\renewcommand{\labelenumi}{\bf \arabic{enumi})}
\item
When $\mu = 0$, $x = 0$ is a fixed point of $f^\cS$,
thus $s'(0;0,\eta,\nu) = 0$,
hence $c_8 = c_9 = c_{10} = 0$.
\item
When $\eta = 0$, $x = \check{x}_0$ is a fixed point of $f^\cS$
and since here $s = \check{s}_0 = 0$, we have
$s'(0;\mu,0,\nu) = 0$,
hence $c_5 = c_7 = 0$.
\item
Similarly when $\nu = 0$, $x = \check{x}_d$ is a fixed point of $f^\cS$.
Here $s = \check{s}_d = \check{u}_d \mu = t_d \mu + O(2)$, thus
$s'(\check{u}_d \mu;\mu,\eta,0) = \check{u}_d \mu$, leading to
\begin{equation}
c_1 = -\frac{c_2}{t_d} \;, \qquad
c_6 = -c_3 t_d \;. \nonumber
\end{equation}
\end{enumerate}
Thus we have reduced the center manifold map (\ref{eq:sPrime1}) to
\begin{equation}
s'(s;\xi) = s - \frac{c_2}{t_d} s^2 + c_2 \mu s + c_3 \eta s + c_4 \nu s
- c_3 t_d \mu \eta + O(3) \;.
\label{eq:sPrime2}
\end{equation}
\begin{enumerate}
\renewcommand{\labelenumi}{\bf \arabic{enumi})}
\addtocounter{enumi}{3}
\item
%When $\mu = 0$, $x = 0$ is a fixed point of $f^\cS$.
%Associated stability multipliers are eigenvalues of $D_x f^\cS(0;0,\eta,\nu)$.
Let $\lambda(\eta,\nu)$ be the eigenvalue of $D_x f^\cS(0;0,\eta,\nu)$,
$C^{K-1}$ dependent on $\eta$ and $\nu$, for which $\lambda(0,0) = 1$.
This eigenvalue is also the stability multiplier
of the fixed point, $s = 0$, of $s'(s;0,\eta,\nu)$,
that is, by (\ref{eq:sPrime2})
\begin{equation}
\lambda(\eta,\nu) =
\frac{\partial s'}{\partial s}(0;0,\eta,\nu) =
1 + c_3 \eta + c_4 \nu + O(2) \;.
\label{eq:lamstabmult}
\end{equation}
Now, $\det(I - D_x f^\cS(0;0,\eta,\nu))$ is equal to the product of all
$N$ eigenvalues (counting algebraic multiplicity) of
$I - D_x f^\cS(0;0,\eta,\nu)$.
Thus
\begin{equation}
\det(I - D_x f^\cS(0;0,\eta,\nu)) = (1 - \lambda(\eta,\nu)) \mathcal{P}(\eta,\nu) \;, \nonumber
\end{equation}
where $\mathcal{P}(\eta,\nu)$ denotes the product of the remaining $N-1$ eigenvalues
of $I - D_x f^\cS(0;0,\eta,\nu)$.
$\mathcal{P}(\eta,\nu)$ is $C^{K-1}$ and $\mathcal{P}(0,0) \ne 0$
since the algebraic multiplicity of the eigenvalue 1 of $M_\cS(0)$ is one.
Let
\begin{equation}
\kappa = \mathcal{P}(0,0) \;, \nonumber
\end{equation}
then, using (\ref{eq:lamstabmult}),
\begin{eqnarray}
\det(I - D_x f^\cS(0;0,\eta,\nu)) & = & (-c_3 \eta - c_4 \nu + O(2))(\kappa + O(1))
\nonumber \\
& = & -c_3 \kappa \eta - c_4 \kappa \nu + O(2) \;. \nonumber
%& = & k_1 \eta + k_2 \nu + O(2)
\end{eqnarray}
Using (\ref{eq:k1k2}), (\ref{eq:exp_res_6}) and (\ref{eq:exp_res_7})
we arrive at
\begin{equation}
c_3 = -\frac{\check{\delta}}{\kappa t_d} \;, \qquad
c_4 = -\frac{\check{\delta}}{\kappa t_{(l-1)d}} \;. \nonumber
\end{equation}
\item
When $\eta = \nu = 0$,
the $\check{\cS}$-cycle has two points on the \sw~and coincides with the
$\hat{\cS}$-cycle.
Both $\check{x}_0(\mu,0,0)$ and $\check{x}_d(\mu,0,0)$ are fixed points of $f^\cS$.
Let $\lambda_1(\mu)$
and $\lambda_2(\mu)$
be the respective eigenvalues of
$D_x f^\cS(\check{x}_0(\mu,0,0);\mu,0,0)$ and
$D_x f^\cS(\check{x}_d(\mu,0,0);\mu,0,0)$,
$C^{K-1}$ dependent on $\mu$ with $\lambda_1(0) = \lambda_2(0) = 1$.
As above, since $\check{s}_0(\mu,0,0) = 0$, from (\ref{eq:sPrime2})
\begin{equation}
\lambda_1(\mu) =
\frac{\partial s'}{\partial s}(\check{s}_0(\mu,0,0);\mu,0,0) =
1 + c_2 \mu + O(\mu^2)) \;, \nonumber
\end{equation}
and then since $\check{s}_d(\mu,0,0) = t_d \mu + O(\mu^2)$,
\begin{eqnarray}
\lambda_2(\mu) =
\frac{\partial s'}{\partial s}(\check{s}_d(\mu,0,0);\mu,0,0) & = &
1 - \frac{2 c_2}{t_d}(t_d \mu + O(\mu^2)) + c_2 \mu + O(\mu^2) \nonumber \\
& = & 1 - c_2 \mu + O(\mu^2)) \;. \nonumber
\end{eqnarray}
Again, as above,
\begin{eqnarray}
\det(I - D_x f^\cS(\check{x}_0(\mu,0,0);\mu,0,0)
& = & -c_2 \kappa \mu + O(\mu^2) \;, \nonumber \\
{\rm and~~~} \det(I - D_x f^\cS(\check{x}_d(\mu,0,0);\mu,0,0)
& = & c_2 \kappa \mu + O(\mu^2) \;. \nonumber
\end{eqnarray}
But, recall (\ref{eq:k0}), so
\begin{equation}
c_2 = -\frac{k_0}{\kappa} \;,
\nonumber
\end{equation}
and by (\ref{eq:m2})
\begin{equation}
\tilde{k_0} = -k_0 \;.
\label{eq:k0relation}
\end{equation}
Substitution of (\ref{eq:k0relation}) into (\ref{eq:phi_2})
verifies (\ref{it:phi_2}) of the lemma
(using also (\ref{eq:exp_res_6}) and (\ref{eq:exp_res_7})).
It now only remains to demonstrate (\ref{it:Lambda}) and (\ref{it:zeta12})
of the lemma.
\end{enumerate}
We have shown that the restriction of $f^\cS$ to $W^c$ is
\begin{equation}
s'(s;\xi) = s + \frac{k_0}{\kappa t_d} s^2
- \frac{k_0}{\kappa} \mu s
- \frac{\check{\delta}}{\kappa t_d} \eta s
- \frac{\check{\delta}}{\kappa t_{(l-1)d}} \nu s
+ \frac{\check{\delta}}{\kappa} \mu \eta + O(3) \;.
\label{eq:sPrimeF}
\end{equation}
We now look for saddle-node bifurcations of (\ref{eq:sPrimeF}).
Since $\frac{\partial s'}{\partial s}(0;0) = 1$ and
$\frac{\partial^2 s'}{\partial s^2}(0;0) \ne 0$,
%\begin{equation}
%\frac{\partial s'}{\partial s}(s;\xi) =
%1 + \frac{2 k_0}{\kappa t_d} s -
%\frac{k_0}{\kappa} \mu -
%\frac{\check{\delta}}{\kappa t_d} \eta -
%\frac{\check{\delta}}{\kappa t_{(l-1)d}} \nu + O(2) \;, \nonumber
%\end{equation}
%which is $C^{K-2}$.
the implicit function theorem implies that
there exists a unique $C^{K-2}$ function, $\psi$ such that
$\frac{\partial s'}{\partial s}(\psi(\xi);\xi) = 1$ for small $\xi$ and
\begin{equation}
\psi(\xi) = \frac{t_d}{2} \mu + \frac{\check{\delta}}{2 k_0} \eta + 
\frac{\check{\delta} t_d}{2 k_0 t_{(l-1)d}} \nu + O(2) \;.
\label{eq:psiProof}
\end{equation}
Then saddle-node bifurcations occur when $s'(\psi(\xi);\xi) = \psi(\xi)$.
Let
\begin{equation}
\Lambda(\xi) = -\frac{4 k_0 \kappa}{\check{\delta}^2 t_d}
\big( s'(\psi(\xi);\xi) - \psi(\xi) \big) \;.
\label{eq:LambdaProof}
\end{equation}
Substitution of (\ref{eq:sPrimeF}) and (\ref{eq:psiProof})
into (\ref{eq:LambdaProof}) yields (\ref{eq:Lambda}).
To complete verification of (\ref{it:Lambda}) of the lemma
we formally show that (\ref{eq:sPrimeF}) has a saddle-node bifurcation
at $\Lambda(\xi) = 0$ whenever $\mu > 0$ by
verifying that all nondegeneracy conditions
of the saddle-node bifurcation theorem \cite{GuHo86}
are indeed satisfied:
\begin{enumerate}
\renewcommand{\labelenumi}{\bf \arabic{enumi})}
\item
by construction, $\displaystyle
\frac{\partial s'}{\partial s}(\psi(\xi),\xi) = 1$
when $\Lambda(\xi) = 0$,
\item
$\displaystyle \frac{\partial s'}{\partial \tilde{\nu}} =
\frac{\check{\delta} t_d}{2 \kappa} \mu + O(2) \ne 0$ when $\mu > 0$
verifying transversality (where $\tilde{\nu}$ is given in (\ref{eq:xiTilde})),
\item
$\displaystyle \frac{\partial^2 s'}{\partial s^2} =
\frac{k_0}{\kappa t_d} + O(1) \ne 0$.
\end{enumerate}

{\bf Step 4:} {\em Compute the intersections of $\Lambda = 0$
with $\eta = 0$ and $\nu = \phi_2(\mu,\eta)$ to obtain
the functions $\zeta_1$ and $\zeta_2$.}\\
To determine where $\Lambda(\xi) = 0$ intersects $\eta = 0$ it is
natural to look at $\Lambda(\mu,0,\nu) = 0$,
but this is insufficient for a derivation of $\zeta_1(\mu)$
because we may not apply the implicit function theorem
to $\Lambda(\xi)$ since it contains no linear terms.
Instead we use the fact that
when $\eta = 0$, the $\check{\cS}$-cycle is also an $\cS$-cycle.
We have $\Lambda = 0$ when, in addition, this periodic solution
has an associated multiplier of 1 as an $\cS$-cycle.
This occurs when (using (\ref{eq:k0}), (\ref{eq:k1k2}) and (\ref{eq:exp_res_6}))
\begin{equation}
\det(I - D_x f^\cS(\check{x}_0(\mu,0,\nu);\mu,0,\nu) =
k_0 \mu + \frac{\check{\delta}}{t_{(l-1)d}} \nu + O(2) = 0
\;, \nonumber
\end{equation}
Application of the implicit function theorem to the previous equation produces
\begin{equation}
\nu = \zeta_1(\mu) = -\frac{k_0 t_{(l-1)d}}{\check{\delta}} \mu + O(\mu^2) \;.
\nonumber
\end{equation}
Moreover,
\begin{equation}
\Lambda(\mu,0,\nu) = \frac{1}{t_{(l-1)d}^2} (\nu - \zeta_1(\mu))^2 +
o(|\nu - \zeta_1(\mu)|^2) \;, \nonumber
\end{equation}
and so for $\mu > 0$, $\Lambda \ge 0$ on the $\nu$-axis.

The curve $(\mu,\zeta_2(\mu),\phi_2(\mu,\zeta_2(\mu))$
along which $\Lambda(\xi) = 0$ intersects the surface $\phi_2(\mu,\eta)$,
is easily computed in a similar fashion.
When $\mu > 0$, $\Lambda \ge 0$ along $\nu = \phi_2(\mu,\eta)$
and so $\Lambda \le 0$ only when $\eta \le 0$ and
$\nu \ge \phi_2(\mu,\eta)$ and stated in the final part of the lemma.
\end{proof}

\begin{proof}[Proof of \Ther~\ref{th:unfold}]

We begin by determining the region of admissibility of
the $\check{\cS}$-cycle, $\{ \check{x}_i \}$.
From (\ref{eq:yidef}), $\check{x}_i(\xi) = \mu (y_i + O(1))$,
thus since by assumption $y_0$ and $y_{ld}$ are the only
points of $\{ y_i \}$ that lie on the \sw~for
small $\xi$ with $\mu > 0$,
here $\check{x}_i(\xi)$
lies on the same side of the \sw~as $y_i$ for each $i \ne 0, ld$.
The $n$-cycle, $\{ y_i \}$, is admissible by assumption (for $\mu > 0$),
thus $\{ \check{x}_i(\xi) \}$ is admissible
exactly when $\check{s}_0, \check{s}_{ld} \ge 0$
(since $\check{\cS}_0 = \check{\cS}_{ld} = \sR$).
By (\ref{eq:cu_0}) and (\ref{eq:cu_ld})
$\check{s}_0, \check{s}_{ld} \ge 0$ when $\eta, \nu \ge 0$, therefore the
$\check{\cS}$-cycle is admissible in $\Psi_1$
as stated in the theorem.

Similarly, for $\mu > 0$ the $\hat{\cS}$-cycle is admissible
exactly when $\hat{s}_0, \hat{s}_{ld} \le 0$
(since $\hat{\cS}_0 = \hat{\cS}_{ld} = \sL$).
By \Lemm~\ref{le:expansions}(\ref{it:phi_1}),
$\hat{s}_{ld}(\xi) = 0$ when $\eta = \phi_1(\mu,\nu)$.
By (\ref{eq:exp_res_4}), (\ref{eq:exp_res_7})
and \Lemm~\ref{le:expansions}(\ref{it:ddFormula}),
$\frac{\partial \hat{s}_{ld}}{\partial \eta}(\xi) =
-\frac{t_{(l-1)d}}{t_{-d}} + O(\xi)$
which is positive for small $\xi$,
hence $\hat{s}_{ld} \le 0$ for $\eta \le \phi_1(\mu,\nu)$
when $\mu > 0$.
Similarly by \Lemm~\ref{le:expansions}(\ref{it:phi_2}),
$\hat{s}_0(\xi) = 0$ when $\nu = \phi_2(\mu,\eta)$
and by (\ref{eq:exp_res_2}), (\ref{eq:exp_res_6})
and \Lemm~\ref{le:expansions}(\ref{it:ddFormula}),
$\frac{\partial \hat{s}_0}{\partial \nu}(\xi) =
-\frac{t_d}{t_{(l+1)d}} + O(\xi)$
which is positive for small $\xi$,
hence $\hat{s}_0 \le 0$ for $\nu \le \phi_2(\mu,\eta)$ when $\mu > 0$.
Therefore the $\hat{\cS}$-cycle is admissible in $\Psi_2$.

It remains to verify admissibility of $\cS$-cycles.
For the theorem in \cite{SiMe09} this was straight-forward
since, if it existed, the $\cS$-cycle was unique.
The situation here is more complicated because
there may be two coexisting, admissible $\cS$-cycles.
In the proof of \Lemm~\ref{le:expansions}
we determined the restriction of $f^\cS$ to
the center manifold through $(s;\xi) = (0;0)$, (\ref{eq:sPrimeF}).
When $\mu > 0$ and $\Lambda(\xi) > 0$, locally
the map (\ref{eq:sPrimeF}) has two distinct fixed points, say,
$s_{0,1}$ and $s_{0,2}$.
We will denote the
corresponding $\cS$-cycles of (\ref{eq:pwsMap}) by
$\{ x_{i,1} \}$ and $\{ x_{i,2} \}$
and assume $s_{0,1} \ge s_{0,2}$.
For small $\xi$ within $\{ \xi ~|~ \mu > 0, \Lambda(\xi) > 0 \}$,
$s_{0,1}$ and $s_{0,2}$ are $C^{K-1}$ functions of $\xi$.

On the surface $\Lambda(\xi) = 0$ the two solutions coincide
(see (\ref{eq:psiProof})):
\begin{equation}
s_{0,1}(\xi) = s_{0,2}(\xi) = \psi(\xi),
{\rm ~when~} \Lambda(\xi) = 0.
\nonumber
\end{equation}
At the intersection of $\Lambda(\xi) = 0$ and $\eta = 0$,
namely $(\mu,0,\zeta_1(\mu))$ (see \Lemm~\ref{le:expansions}(\ref{it:zeta12})),
$s_{0,1} = s_{0,2} = 0$.
Thus by (\ref{eq:psiProof}), on $\Lambda(\xi) = 0$,
$s_{0,1} = s_{0,2} < 0$ when $\nu < \zeta_1(\mu)$ and
$s_{0,1} = s_{0,2} > 0$ when $\nu > \zeta_1(\mu)$.

Now, for $\mu > 0$, $s_{0,1}$ and $s_{0,2}$ can only be zero if
$\eta = 0$ because if $s=0$ is a fixed point of (\ref{eq:sPrimeF})
then the corresponding $\cS$-cycle would also be an $\check{\cS}$-cycle
which for $\mu > 0$ must be $\{ \check{x}_i \}$,
so then $\check{s}_0(\xi) = 0$ and by (\ref{eq:cu_0})
we would necessarily have $\eta = 0$.
Consequently one of $s_{0,1}$ and $s_{0,2}$ is zero when $\eta = 0$
and $s_{0,1}$ and $s_{0,2}$ are both nonzero when $\eta \ne 0$.
Since we assume $s_{0,1} > s_{0,2}$ for $\Lambda(\xi) \ne 0$,
when $\mu > 0$ and $\eta = 0$ we must have
$s_{0,1} = 0$ when $\nu \le \zeta_1(\mu)$ and
$s_{0,2} = 0$ when $\nu \ge \zeta_1(\mu)$.
Consequently $s_{0,1} < 0$ in $\Psi_2 \cup \Psi_3$
and $s_{0,2} < 0$ in $\Psi_1 \cup \Psi_2 \cup \Psi_3$.

Before we are able to perform a similar analysis of $s_{i,j}$ for $i \ne 0$,
we find it necessary to
first derive an expression for $s_{i,j}$ in terms of $t_i$ and $t_{d+i}$.
Recall that the center manifold, $W^c$,
is given by (\ref{eq:Hproof}) where
$M_\cS(0) v = v$ and $e_1^{\sf T} v = 1$.
When $\xi = 0$, $y_0$ and $y_d$ are both fixed points of $h^\cS$,
thus $(I-M_\cS(0)) y_0 = P_\cS(0) b(0) = (I-M_\cS(0)) y_d$ and so
\begin{equation}
(I - M_\cS(0)) (y_0 - y_d) = 0 \;,
\nonumber
\end{equation}
and since $y_0 \ne y_d$ (\Lemm~\ref{le:shrPoint}) and
the eigenvalue $1$ of the matrix $M_\cS(0)$ has algebraic multiplicity one,
$y_0 - y_d$ is a scalar multiple of $v$.
Due to the specified vector scaling we have
\begin{equation}
v = \frac{1}{t_d}(y_d - y_0) \;.
\label{eq:vFormula}
\end{equation}
Combining (\ref{eq:Hproof}) and (\ref{eq:vFormula}) yields
\begin{equation}
x_{0,j}(\xi) = \left( \mu - \frac{s_{0,j}(\xi)}{t_d} \right) y_0 +
\frac{s_{0,j}(\xi)}{t_d} y_d + O(2) \;.
\label{eq:x0j}
\end{equation}
This may be generalized to an expression for $x_{i,j}(\xi)$ using
$x_{i+1,j} = \mu b + A_{\cS_i} x_{i,j} + O(2)$
and $y_{i+1} = \mu b(0) + A_{\cS_i}(0) y_i$, from which we deduce
\begin{equation}
s_{i,j}(\xi) = \left( \mu - \frac{s_{0,j}(\xi)}{t_d} \right) t_i +
\frac{s_{0,j}(\xi)}{t_d} t_{d+i} + O(2) \;,
\label{eq:sij}
\end{equation}
and hence
\begin{equation}
s_{i,1}(\xi) - s_{i,2}(\xi) = -\frac{1}{t_d}
(t_i - t_{d+i}) (s_{0,1}(\xi) - s_{0,2}(\xi)) + O(2) \;.
\nonumber
\end{equation}
Therefore for small $\xi > 0$
with $\mu > 0$ and $\Lambda(\xi) > 0$,
$s_{i,1} > s_{i,2}$ if $t_i > t_{d+i}$ and 
$s_{i,1} < s_{i,2}$ if $t_i < t_{d+i}$
(because we assumed $s_{0,1} > s_{0,2}$).

Above we showed that when
$\eta = \nu = 0$ and $\mu > 0$, $s_{0,1} = 0$ and $s_{0,2} < 0$.
Thus here $s_{ld,1} = 0$ and $s_{ld,2} > 0$
and by \Lemm~\ref{le:expansions}, along $\eta = \phi_1(\mu,\nu)$,
$s_{ld,1} = 0$ and $s_{ld,2} > 0$.
From this is easily follows that
$s_{ld,1} > 0$ in $\Psi_2 \cup \Psi_3$
and $s_{ld,2} > 0$ in $\Psi_1 \cup \Psi_2 \cup \Psi_3$.
Similarly when $\nu = 0$ and $\mu > 0$,
$s_{(l-1)d,1} < 0$ and $s_{(l-1)d,2} = 0$ and consequently
$s_{(l-1)d,1} < 0$ in $\Psi_1 \cup \Psi_2 \cup \Psi_3$ and
$s_{(l-1)d,2} < 0$ in $\Psi_1 \cup \Psi_3$.
By analogous arguments,
$s_{-d,1} > 0$ in $\Psi_1 \cup \Psi_2 \cup \Psi_3$ and
$s_{-d,2} > 0$ in $\Psi_1 \cup \Psi_3$.
By (\ref{eq:sij}), for $i \ne 0, (l-1)d, ld, -d$,
$s_{i,1}$ and $s_{i,2}$ have the desired sign for admissibility 
for small $\xi$ with $\mu > 0$.
The above statements show that
$\{ x_{i,1} \}$ is admissible in $\Psi_2 \cup \Psi_3$
and $\{ x_{i,2} \}$ is admissible in $\Psi_1 \cup \Psi_3$
which completes the proof.
\end{proof}

%\bibliographystyle{unsrt}
%\bibliography{../PhDThesis}

\end{document}